\documentclass[11pt]{amsart}

\usepackage{amsmath,amsthm,amsfonts,amssymb,bm,graphicx}
\usepackage{mathdots}

\usepackage
{hyperref}
\hypersetup{colorlinks=true,citecolor=blue,linkcolor=blue,urlcolor=blue,
pdfstartview=FitH, pdfauthor=Liyang Zhang ,pdftitle=Quantum Ergodicity of Degenerate Eisenstein Series on \texorpdfstring{$GL(n)$}{Lg}}

\newcommand{\bbZ}{\mathbb{Z}}
\newcommand{\bbC}{\mathbb{C}}

\newcommand{\bbH}{\mathbb{H}}
\newcommand{\bbR}{\mathbb{R}}

\newcommand{\bemone}{\left(\begin{array}{c}}
\newcommand{\bemtwo}{\left(\begin{array}{cc}}
\newcommand{\bemthree}{\left(\begin{array}{ccc}}
\newcommand{\bemfour}{\left(\begin{array}{cccc}}
\newcommand{\bemfive}{\left(\begin{array}{ccccc}}
\newcommand{\bemsix}{\left(\begin{array}{cccccc}}
\newcommand{\bemseven}{\left(\begin{array}{ccccccc}}
\newcommand{\enm}{\end{array}\right)}
\newcommand{\re}{\mbox{Re}}

\newcommand{\doublesum}{\mathop{\sum\sum}}

\newcommand{\psum}{\mathop{{\sum}'}}

\newcommand{\lbar}{\left|}
\newcommand{\rbar}{\right|}
\newcommand{\lbra}{\left\{}
\newcommand{\rbra}{\right\}}
\newcommand{\lp}{\left(}
\newcommand{\rp}{\right)}

\newcommand{\sltz}{SL_2(\bbZ)}

\newcommand{\wjac}{W_{\mbox{\tiny{Jac}}}}
\newcommand{\quo}{SL_n(\bbZ)\backslash X_n}

\newcommand{\beq}{\begin{eqnarray*}}
\newcommand{\eeq}{\end{eqnarray*}}
\newcommand{\be}{\begin{eqnarray}}
\newcommand{\ee}{\end{eqnarray}}
\theoremstyle{plain}
 \newtheorem{thm}{Theorem}[section]
 \newtheorem{prop}{Proposition}[section]
 \newtheorem{lem}{Lemma}[section]
 \newtheorem{cor}{Corollary}[section]
\theoremstyle{definition}
 
 \newtheorem{dfn}{Definition}[section]
\theoremstyle{remark}
 \newtheorem{rem}{Remark}[section]
 \numberwithin{equation}{section}

\makeatletter
\DeclareRobustCommand\widecheck[1]{{\mathpalette\@widecheck{#1}}}
\def\@widecheck#1#2{%
    \setbox\z@\hbox{\m@th$#1#2$}%
    \setbox\tw@\hbox{\m@th$#1%
       \widehat{%
          \vrule\@width\z@\@height\ht\z@
          \vrule\@height\z@\@width\wd\z@}$}%
    \dp\tw@-\ht\z@
    \@tempdima\ht\z@ \advance\@tempdima2\ht\tw@ \divide\@tempdima\thr@@
    \setbox\tw@\hbox{%
       \raise\@tempdima\hbox{\scalebox{1}[-1]{\lower\@tempdima\box
\tw@}}}%
    {\ooalign{\box\tw@ \cr \box\z@}}}
\makeatother

\begin{document}

\author{Liyang Zhang}
\address{10 Hillhouse Ave. New Haven, CT 06511} \email{liyang.zhang@yale.edu}

\title{Quantum Unique Ergodicity of Degenerate Eisenstein Series on $GL(n)$}
 
\thanks{The author was supported by the following grants from Alex Kontorovich: NSF CAREER grant DMS-1254788 and DMS-1455705}


\begin{abstract}  
We prove quantum unique ergodicity for a subspace of the continuous spectrum spanned by the degenerate Eisenstein Series on $GL(n)$.
\end{abstract}

\setcounter{tocdepth}{2}  \maketitle 

\tableofcontents

\section{Introduction}

\subsection{Introduction}
In the classical setting, the evolution of a dynamical system $(X,\mu,T)$ can be described by the geodesic flow:
\beq
g_t: T^*X\rightarrow T^*X
\eeq
where $g_t {\bf x}_0={\bf x}_t$ is given by the Hamiltonian. We say the system is ergodic if for every $f\in L_\mu^2$ and for almost every starting position ${\bf x}_0$ the time average becomes the spatial average:
\beq
\lim_{S\rightarrow\infty} \frac{1}{S}\int_0^S f(g_t({\bf x}_0))dt = \frac{1}{\mu(X)}\int_{X}f({\bf x})d{\bf x}.
\eeq
A classic example of an ergodic dynamical system is the Bunimovich Stadium \cite{Bu}\cite{BS}.

In the quantum setting, the evolution of a system is governed by the Schr\"odinger equation:
\beq
-\frac{\hbar^2}{2m}\Delta \psi_n = \lambda_n\psi_n
\eeq
where $\Delta$ is the Laplacian. A system is quantum uniquely ergodic if in the semiclassical limit ($\hbar\rightarrow 0$), each individual eigenfunction $|\psi_n|^2$ become equi-distributed.

Let $(M,\mu)$ be a Riemannian manifold with laplacian $\Delta$ and let $\psi_n$ a set of orthonormal eigenfunctions of $\Delta$ with corresponding eigenvalues $0\le\lambda_1\le\lambda_2\le\ldots$. Schnirelman \cite{Sh}, Colin de Verdi\'ere \cite{Co} and Zelditch \cite{Ze1} proved quantum ergodicity: eigenfunctions of $\Delta$ become equi-distributed with respect to the volume measure in the high energy limit along a subsequence $n_k$ of density one. Zelditch \cite{Ze2} extended this result to the modular surface, which is not compact. Hejhal-Rackner \cite{HR}, Rudnick-Sarnak \cite{RS} conjectured that there are no exceptional subsequences, that is, there is quantum unique ergodicity. The arithmetic version of the conjecture was famously proved by Lindenstrauss \cite{Li} and Soundararajan \cite{S}. For higher rank, Silberman and Venkatesh formualted and proved some cases of quantum unique ergodicity for compact locally symmetric spaces in \cite{SV1} and \cite{SV2}.

For non-compact quotients, the Laplacian $\Delta$ has continuous spectrum and one can formulate QUE for such. For example, on the modular surface the continuous spectrum is spanned by the Eisenstein series
\beq
E(z,s) = \sum_{\gamma\in\Gamma_\infty\backslash SL(2,\bbZ)} \Im(\gamma z)^s.
\eeq
Arithmetic quantum unique ergodicity of Eisenstein series on $GL(2)$ was first formulated and proved by Luo and Sarnak in \cite{LS}. The unitary Eisenstein series $E(z,\frac{1}{2}+it)$ is an eigenfunction of the hyperbolic Laplacian with eigenvalue $\frac{1}{4}+t^2$. Define $d \nu_t = \lbar E\lp z,\frac{1}{2}+it\rp\rbar^2 \frac{dxdy}{y^2}$. Then for compact Jordan measurable sets $A,B$ in $SL(2,\bbZ)\backslash SL(2,\bbR) /SO(2,\bbR)$, Luo and Sarnak showed 

\beq
\lim_{t\rightarrow\infty }\frac{\nu_t(A)}{\nu_t(B)} = \frac{\mbox{Vol}(A)}{\mbox{Vol}(B)}
\eeq
which represents arithmetic quantum unique ergodicity for the continuous spectrum. More precisely, they showed
\beq
\nu_t(A)= \frac{12}{\pi}\mbox{Vol}(A) \log t + O(1)
\eeq
as $t\rightarrow \infty$. Note that ergodic methods typically do not give rates. We also remark that the constant $\frac{12}{\pi}$ is different from that in \cite{LS} because of different normalization. 

In this paper, we study the analog of Luo and Sarnak's result on a subspace of the $GL(n)$ continuous spectrum spanned by the degenerate Eisenstein series induced from the maximal parabolic subgroup with the constant function $E_{n-1,1}\lp z,s,1\rp$. Eisenstein series are eigenfunctions of Casimir operators and on $GL(n)$, the analog of the hyperbolic Laplacian is $\Delta_{2,n}$ (see section 2) where
\beq
\lp\Delta_{2,n} +\lp\frac{n^2-n}{8}+\frac{n^2-n}{2} t^2\rp \rp E_{n-1,1}\lp z,\frac{1}{2}+it,1\rp =0.
\eeq 
Define $\mu_{n,t} = \lbar E_{n-1,1}(z,\frac{1}{2}+it,1)\rbar^2d^*z$ where $d^*z$ is the Haar measure. We will show the following:

\begin{thm}\label{thm1} Let $A,B$ be compact Jordan measurable subsets of $SL(n,\bbZ)\backslash SL(n,\bbR) /SO(n,\bbR)$. Then for $n\ge2$,
\beq
\lim_{t\rightarrow\infty }\frac{\mu_{n,t}(A)}{\mu_{n,t}(B)} = \frac{\mbox{Vol}(A)}{\mbox{Vol}(B)}.
\eeq
More precisely,
\beq
\mu_{n,t}(A) =  \frac{2}{\xi(n)} \mbox{Vol}(A) \log t + O_A(1)
\eeq
as $t\rightarrow \infty$.
\end{thm}

\begin{rem} This theorem will serve as a stepping stone towards proving quantum unique ergodicity of higher rank non-degenerate Eisenstein series. QUE of higher rank non-degenerate Eisenstein series is a much more difficult problem as it is related to the shifted convolution problem involving generalized divisor functions and Fourier coefficients of higher rank Maass cusp forms. 
\end{rem}

\subsection{Strategy for proof of Theorem \texorpdfstring{\ref{thm1}}{Lg}}

We use the spectral decomposition of $\frak{L}^2(SL_n(\bbZ)\backslash X_n)$ to divide the proof of Theorem \ref{thm1} into evaluating the following three type of integrals (see Theorem \ref{prop41}):
\beq
&&\int_{SL_n(\bbZ)\backslash X_n}\phi(z)\lbar E_{n-1,1}\lp z,\frac{1}{2}+it,1\rp\rbar^2d^*z,\\
&&\int_{SL_n(\bbZ)\backslash X_n}E_{1,\ldots,1}(z,\eta) \lbar E_{n-1,1}\lp z,\frac{1}{2}+it,1\rp\rbar^2d^*z,\\
&&\int_{SL_n(\bbZ)\backslash X_n}E_{n_1,\ldots,n_r}(z,\psi,u_1,\ldots,u_r) \lbar E_{n-1,1}\lp z,\frac{1}{2}+it,1\rp\rbar^2d^*z.
\eeq
Here $\phi(z)$ is a $GL(n)$ cusp form, $E_{1,\ldots,1}(z,\eta)$ is an incomplete Eisenstein series associated to the minimal parabolic subgroup, and $E_{n_1,\ldots,n_r}(z,\psi,u_1,\ldots,u_r)$ is an incomplete Eisenstein series associated to the $(n_1,\ldots,n_r)$ parabolic subgroup induced from cusp forms $u_j$ on $GL(n_j)$. These integrals will be evaluated in sections 4 and 5.

\textbf{Acknowledgement.} I would like to thank Alex Kontorovich for suggesting the problem, many enlightening discussions, and careful readings of the various versions of this paper. I also would like to thank Valentin Blomer for helpful comments.
\newpage
\section{Automorphic Forms on \texorpdfstring{$GL(n)$}{Lg}}
 
 Here we give a brief introduction to automorphic forms on $GL(n)$, make some preliminary computations and set the notation for the rest of the paper.
 
 \subsection{Automorphic Functions, Automorphic Forms, and Fourier Expansion}

\vspace{10mm}
 
We are working over the generalized upper half space $ X_n :\ = GL_n(\bbR)\slash \lp O_n(\bbR) \cdot \bbR^*\rp$. By the Iwasawa decomposition (see section 1.2 of \cite{Go}), $X_n$ consists of matrices of the form $z= x\cdot y$ with
 
\begin{eqnarray}
x = \bemfive 1 & x_{1,2}& x_{1,3}&\cdots& x_{1,n}\\  &1&x_{2,3}&\cdots & x_{2,n} \\ &&\ddots&&\vdots\\ &&&1& x_{n-1,n}\\&&&&1\enm, \ \ \ y=   \bemfive   y_1y_2\cdots y_{n-1} &&&&\\  & y_1y_2\cdots y_{n-2}&&&\\ &&\ddots&&\\ &&&y_1&\\&&&&1\enm. \label{coor}
\end{eqnarray}

 where $x_{i,j}\in\bbR$ and $y_k>0$. We say a function $f$ is automorphic if 
 \beq
 f(\gamma z) = f(z)\hspace{10mm} \forall \gamma\in SL_n(\bbZ),\forall z\in X_n.
 \eeq

The space $X_n$ is equipped with a left $GL_n(\bbR)$-invariant measure $d^*z$ on $X_n$ given explicitly by
\beq
d^*z = c_n \ d^*x \ d^*y
\eeq
where
\beq
c_n = n^{-1} \prod_{\ell=2}^n\xi(\ell)^{-1},\hspace{5mm} d^*x = \prod_{1\le i<j\le n}d x_{i,j},\hspace{5mm} \prod_{k=1}^{n-1}d_k^{-k(n-k)}\frac{dy_k}{y_k}.
\eeq
Throughout this paper, we define the completed zeta-function $\xi(s) $ to be $\pi^{-\frac{s}{2}}\Gamma\lp\frac{s}{2}\rp\zeta(s)$. The normalization for this measure is chosen so that
\beq
\int_{\quo} 1 \ d^*z = 1.
\eeq

Our interest lies in the Hilbert space $\frak{L}(\quo)$ where the inner product is defined by
\beq
\langle f,g\rangle = \int_{\quo} f(z)\overline{g}(z)\ d^*z.
\eeq
For a smooth function $f$ in $\frak{L}(\quo)$, standard Fourier theory with the automorphy of $f$ give rise to a Fourier expansion of $f$. 
\begin{thm}(See section 5.3 of \cite{Go}.)
Let $P_{n-1,1}(\bbZ) = \lbra \bemfive &&*&&\\ 0&0&\cdots&0&0\enm\in SL_{n}(\bbZ)\rbra$. For $f\in \frak{L}(\quo)$ a smooth function, we have for all $z\in\quo$
\begin{eqnarray}
f(z) =\left. \sum_{m_1=0}^\infty\sum_{m_2=0}^\infty\psum_{\gamma_2\in P_{1,1}\backslash SL_2(\bbZ)}\cdots\sum_{m_{n-1}=0}^\infty\psum_{\gamma_{n-1}\in P_{n-2,1}\backslash SL_{n-1}(\bbZ)}\hat{f}_{(m_1,\ldots,m_{n-1})}\lp z\rp\rbar_{\gamma_2\cdots\gamma_{n-1}}. \label{autofourier}
\end{eqnarray}
where the slash operator is defined by \beq
\left.g(z)\rbar_{\gamma_2\cdots\gamma_{n-1}} = g\lp \bemtwo \gamma_2&\\ &I_{n-2}\enm\cdot\cdots\cdot\bemtwo \gamma_{n-1}&\\ &1\enm\cdot z\rp,
\eeq
and for each $2\le h\le n-1$, the primed summation over each $P_{h-1,1}\backslash SL_{h}(\bbZ)$ is summed only if $m_{h}\neq0$.
Let $U_n(\bbZ)$ $\lp U_n(\bbR)\rp$ denote the group of $n\times n$ upper triangular matrices with integer (real) entries and $1$'s on the diagonal \beq
\hat{f}_{(m_1,\ldots,m_{n-1})}(z) = \int_0^1\cdots \int_0^1 \phi(u\cdot z) e(-m_1u_{1,2}-m_2u_{2,3} -\cdots - m_{n-1}u_{n-1,n})\prod _{1\le i<j\le n} d{u_{i,j}}
\eeq
and 
\beq
u = \bemfive 1 & u_{1,2}& u_{1,3}&\cdots& u_{1,n}\\  &1&u_{2,3}&\cdots & u_{2,n} \\ &&\ddots&&\vdots\\ &&&1& u_{n-1,n}\\&&&&1\enm \in U_n(\bbR).
\eeq
\end{thm}

Our goal is to describe a spectral decomposition of the space $\frak{L}(\quo)$. It is natural to turn to the theory of differential operators. Let $\frak{D}_n$ be the center of the universal enveloping algebra of the Lie algebra $\frak{gl}_n(\bbR)$. Let $E_{i,j}\in \frak{gl}_n(\bbR)$ denote the matrix with a $1$ at the $(i,j)$ entry and zeros elsewhere. We define the differential operator $D_{i,j}$ by
\beq
(D_{i,j}f)(g) =\left. \frac{\partial}{\partial t}f\lp g\cdot \exp(tE_{i,j})\rp\rbar_{t=0}
\eeq
for a smooth function $f:GL_n(\bbR)\mapsto\bbC$.
\begin{prop}(See section 2.3 of \cite{Go}.)
For $n\ge 2$ and $2\le m\le n$, the differential operators (Casimir operators)
\beq
\Delta_{m,n}:= \sum_{i_1=1}^n\sum_{i_2=1}^n\cdots\sum_{i_m=1}^n D_{i_1,i_2}\circ D_{i_2,i_3}\circ\cdots\circ D_{i_m,i_1}
\eeq
generate $\frak{D}_n$ as a polynomial algebra of rank $n-1$. 
\end{prop}

For $n\ge2$ and $\nu=(\nu_1,\ldots,\nu_{n-1})\in\bbC^{n-1}$, the $I$-function defined by 
\beq
I_{\nu}(z) = \prod_{i=1}^{n-1}\prod_{j=1}^{n-1} y_i^{b_{i,j}\nu_j},
\eeq
where
\beq
b_{i,j}  = \lbra \begin{array}{ll} ij& \mbox{if }i+j\le n,\\ (n-i)(n-j) & \mbox{if }i+j\ge n \end{array}\right.
\eeq
 is an eigenfunction of all differential operators in $\frak{D}_n$. Let $\lambda_{m,n}$ be the corresponding eigenvalues, i.e.
 \beq
\Delta_{n,m}I_{\nu}(z)= \lambda_{m,n}I_{\nu}(z).
 \eeq 
 \begin{dfn}
An automorphic form $\phi(z)\in\mathcal{L}^2(SL_n(\bbZ)\backslash X_n)$ of spectral type $\nu\in\bbC^{n-1}$ is a smooth function satisfying: 
 \begin{enumerate}
 \item $\phi(\gamma z) = \phi(z)$, \ \ \  $\forall \ \ \gamma \in SL_n(\bbZ) ,\forall z\in X_n$,
 \item  $\Delta_{m,n} \phi(z) = \lambda_{m,n} \phi(z)$;
  \end{enumerate}
if $\phi(z)$ also satisfies
\beq
\int_{(SL_n(\bbZ)\cap U)\backslash U} \phi(u\cdot z) du=0,
\eeq
for all matrices of the form 
\beq
U = \lbra \bemfive I_{r_1}&&&&\\ & I_{r_2} &&*&\\ &&\ddots&&\\ &&&& I_{r_m}\enm\rbra \subset  SL_n(\bbR)
\eeq
with $r_1+\cdots +r_m = n$, then $\phi(z)$ is a Maass form.

\end{dfn}

For an automorphic form $\phi(z)$ of type $\nu$, the Fourier expansion (\ref{autofourier}) can be made more explicit by the theory of Whittaker functions. Multiplicity one theorem implies that each $\hat{\phi}_{(m_1,\ldots,m_{n-1})}(z)$ can be expressed in terms of Whittaker functions:
 
\beq
W_{(m_1,\ldots,m_{n-1})}^\nu(z,w):=\int_{U_n(\bbR)} I_\nu(w\cdot u\cdot z) e(-m_1u_{1,2}-\cdots-m_{n-1}u_{n-1,n}) \prod _{1\le i<j\le n} d{u_{i,j}},
\eeq
 where $w\in W_n$ is an element of the Weyl group. If $\phi(z)$ is a Maass cusp form, then (\ref{autofourier}) can be further simplified to
 \beq
\phi(z) = \sum_{\gamma\in U_{n-1}(\bbZ)\backslash SL_{n-1}(\bbZ)}\sum_{m_1=1}^\infty\cdots\sum_{m_{n-2}=1}^\infty \sum_{m_{n-1}\neq 0}\left. a_{(m_1,\ldots,m_{n-1})}W_{(m_1,\ldots,m_{n-1})}^\nu(z,w_l) \right|_\gamma,
\eeq
where $w_l$ is the long Weyl element $\bemfour &&&1\\ &&1&\\ &\iddots&&\\ 1&&&\enm$.
 \vspace{20mm}
\subsection{Parabolic Subgroups and Eisenstein Series}

Here we give a summary of Langlands' theory of Eisenstein series associated to Maass forms.

\begin{dfn}
The standard parabolic subgroup $P_{n_1,\ldots,n_r}(\bbR)\subset GL_n(\bbR)$ associated to the partition $n=n_1+n_2+\cdots+n_r$ is defined to be the group of matrices of the form
\beq
\bemfour \frak{m}_{n_1}&*&\cdots&*\\ 0&\frak{m}_{n_2}&\cdots&* \\ \vdots &\vdots&\ddots&\vdots\\0&0&\cdots&\frak{m}_{n_r}\enm,
\eeq
where $\frak{m}_{n_i}\in GL_{n_i}(\bbR)$ for $1\le i\le r$. We also define $P_{n_1,\ldots,n_r}(\bbZ) = P_{n_1,\ldots,n_r}(\bbR)\cap SL_{n}(\bbZ)$. Two parabolic subgroups $P_{n_1,\ldots,n_r}(\bbR),P_{n'_1,\ldots,n'_r}(\bbR)$ of $GL_n(\bbR)$ are said to be associate if the set $\{n_1,\ldots,n_r\}$ is a permutation of the set $\{n'_1,\ldots,n'_r\}$.
\end{dfn}
A parabolic subgroup $P_{n_1,\ldots,n_r}(\bbR)$ can be decomposed into 
\beq
P_{n_1,\ldots,n_r}(\bbR) = N_{n_1,\ldots,n_r}(\bbR)\ M_{n_1,\ldots,n_r}(\bbR),
\eeq
where 
\beq
N_{n_1,\ldots,n_r}(\bbR) = \lbra\bemfour I_{n_1}&*&\cdots&*\\ 0&I_{n_2}&\cdots&*\\ \vdots &\vdots&\ddots&\vdots\\0&0&\cdots&I_{n_r}\enm\in GL_n(\bbR),I_{k}\mbox{ is an $k\times k$ identity matrix }\rbra
\eeq
is the unipotent radical and
\beq
M_{n_1,\ldots,n_r}(\bbR) = \lbra \bemfour \frak{m}_{n_1}&0&\cdots&0\\ 0&\frak{m}_{n_2}&\cdots&0 \\ \vdots &\vdots&\ddots&\vdots\\0&0&\cdots&\frak{m}_{n_r}\enm, \frak{m}_k\in GL_{n_k}(\bbR)\rbra
\eeq
is the Levi component. This is the Langlands decomposition of parabolic subgroups. We define $N_{n_1,\ldots,n_r}(\bbZ) = N_{n_1,\ldots,n_r}(\bbR)\cap SL_{n}(\bbZ),M_{n_1,\ldots,n_r}(\bbZ) = M_{n_1,\ldots,n_r}(\bbR)\cap SL_{n}(\bbZ)$. For $g\in P_{n_1,\ldots,n_r}(\bbR)$, Langlands decomposition naturally gives rise to the projection maps $\frak{m}_{n_i}:P_{n_1,\ldots,n_r}(\bbR)\mapsto GL_{n_i}(\bbR)$ by
\beq
g = \bemfour I_{n_1}&*&\cdots&*\\ 0&I_{n_2}&\cdots&*\\ \vdots &\vdots&\ddots&\vdots\\0&0&\cdots&I_{n_r}\enm\cdot \bemfour \frak{m}_{n_1}(g)&0&\cdots&0\\ 0&\frak{m}_{n_2}(g)&\cdots&0 \\ \vdots &\vdots&\ddots&\vdots\\0&0&\cdots&\frak{m}_{n_r}(g)\enm.
\eeq
To describe Eisenstein series associated to Maass forms we also need to define $I$-functions associated to parabolic subgroups. 
\begin{dfn}
Let $s= (s_1,\ldots,s_r)\in\bbC^r$ satisfying $\sum_{i=1}^r n_i s_i=0$ and for $z\in X_n$ in Iwasawa form (\ref{coor}) we define
\beq
I_s(z,P_{n_1,\ldots,n_r}) = \lp \prod_{j_1=n-n_{1}+1}^n Y_{j_1}\rp^{s_1} \cdot \lp \prod_{j_2=n-n_{1}-n_{2}+1}^{n-n_1} Y_{j_2}\rp^{s_2}\cdot \cdots\cdot  \lp \prod_{j_r=1}^{n_r} Y_{j_r}\rp^{s_r},
\eeq
where $Y_1,Y_2,\ldots, Y_n$ are defined by
\beq
\bemfour  Y_n&&&\\ &Y_{n-1}&&\\&&\ddots&\\ &&& Y_1 \enm= \bemfive  y_1y_2\cdots y_{n-1}\\&y_1y_2\cdots y_{n-2}\\ &&\ddots \\ &&& y_1\\ &&&&1\enm.
\eeq
\end{dfn}

\begin{dfn}
Let $P_{n_1,\ldots,n_r}(\bbR)$ be a parabolic subgroup of $GL_{n}(\bbR)$ with projection maps $\frak{m}_{n_i}$ defined by Langlands decomposition. Let $\phi_i$ be Maass forms on $GL_{n_i}(\bbR)$ for $1\le i\le r$, and let $s=(s_1,\ldots,s_r)\in\bbC^r$ satisfy 
\beq
\sum_{i=1}^r n_is_i = 0.
\eeq
We define the Eisenstein series $E_{(n_1,\ldots,n_r)}\lp z,s,\phi_1,\ldots,\phi_r\rp$ by
\beq
E_{(n_1,\ldots,n_r)}\lp z,s,\phi_1,\ldots,\phi_r\rp = \sum_{\gamma\in P_{n_1,\ldots,n_r}(\bbZ)\backslash SL_n(\bbZ)} \prod_{i=1}^r\phi_i(\frak{m}_i(\gamma z))I_s(\gamma z,P_{n_1,\ldots,n_r}).
\eeq
Because of the relation $\sum_{i=1}^r n_is_i = 0$, we may eliminate $s_r$. So $E_{n_1,\ldots,n_r}\lp z,s,\phi_1,\ldots,\phi_r\rp$ is really a function of $z$ and $s=(s_1,\ldots,s_{r-1})\in\bbC^{r-1}$ and we will use this convention for the remaining part of this paper. We want to especially point out that if we use the partition $n=1+\cdots+1$, we get the minimal parabolic Eisenstein series
\beq
E_{(1,\ldots,1)}(z,s) =\sum_{\gamma\in P_{1,\ldots,1}(\bbZ)\backslash SL_n(\bbZ)}   I_s(\gamma z).
\eeq
For the maximal parabolic $P_{n-1,1}$, we can define the totally degenerate Eisenstein series 
\beq
E_{(n-1,1)}(z,s,1) = \sum_{\gamma\in P_{n-1,1}(\bbZ)\backslash SL_n(\bbZ)}  I_s(\gamma z,P_{n-1,1}).
\eeq
This is the main object of study in this paper.
\end{dfn}

Langlands' theory of constant terms along arbitrary parabolic subgroups play a central role in understanding these Eisenstein series. We will use this theory to understand the Fourier expansion of some Eisenstein series.

\begin{prop}
\label{constant}
Let $n=n_1+\ldots+n_r$ be a partition in descending order with $n_r\ge2$, and let $E_{n_1,\ldots,n_r}\lp z,s,\phi_1,\ldots,\phi_r\rp$ be an Eisenstein series. Then its constant terms along $P_{(n-1,1)}$, $P_{(n-3,1,2)}$, $P_{(1,\ldots,1)}$ are zero:
\beq
\int_{N_{n-1,1}(\bbZ)\backslash N_{n-1,1}(\bbR)} E_{(n_1,\ldots,n_r)}\lp z,s,\phi_1,\ldots,\phi_r\rp =0,\\
\int_{N_{n-3,1,2}(\bbZ)\backslash N_{n-3,1,2}(\bbR)} E_{(n_1,\ldots,n_r)}\lp z,s,\phi_1,\ldots,\phi_r\rp =0,\\
\int_{N_{1,\ldots,1}(\bbZ)\backslash N_{n-1,1}(\bbR)} E_{(n_1,\ldots,n_r)}\lp z,s,\phi_1,\ldots,\phi_r\rp =0.
\eeq
\end{prop}
\begin{proof}
This is a direct result of the proposition in II.1.7 of \cite{MW}. Since $n_i\ge 2$ for all $1\le i\le r$, $w M_{n_1,\ldots,n_r}(\bbR)w^{-1}$ cannot be contained in either $M_{(n-1,1)}(\bbR)$, $M_{(n-3,1,2)}(\bbR)$ or $M_{(1,\ldots,1)}(\bbR)$ for any Weyl element $w$. 
\end{proof}

\begin{cor} \label{coezero}
Let $n=n_1+\ldots+n_r$ be a partition in descending order with $n_r\ge2$, then in the Fourier expansion of $E_{(n_1,\ldots,n_r)}\lp z,s,\phi_1,\ldots,\phi_r\rp$, the Fourier coefficients $a_{(m_1,\ldots,m_{n-2},0)}$, $ a_{(m_1,\ldots m_{n-4},0,0,m_{n-1})}$ and $a_{(0,\ldots,0)}$ are all zero.
\end{cor}
\begin{proof}
Let $f(z)E_{(n_1,\ldots,n_r)}\lp z,s,\phi_1,\ldots,\phi_r\rp$. We have
\beq
\hat{f}_{(m_1,\ldots,m_{n-2},0)}(z) &=& \int_0^1\cdots \int_0^1 f(u\cdot z) e(-m_1u_{1,2}-m_2u_{2,3} -\cdots - m_{n-2}u_{n-2,n-1})\prod _{1\le i<j\le n} d{u_{i,j}}\\
&=& \int_0^1\cdots \int_0^1 f\lp \bemfive 1&0& 0&\cdots& u_{1,n}\\ &1&0&\cdots&u_{2,n} \\ &&\ddots&\cdots&\vdots\\ &&&1&u_{n-1,n}\\&&&&1\enm\cdot \bemfive 1&u_{1,2}&\cdots&u_{1,n-1}&0\\ &1&\cdots& u_{2,n-1}&0 \\ &&\ddots&\vdots&\vdots \\ &&&1&0\\ &&&&1\enm\cdot z\rp \prod _{1\le i\le n-1} d{u_{i,n}}
\\ &&\times e(-m_1u_{1,2}-m_2u_{2,3} -\cdots - m_{n-2}u_{n-2,n-1})\prod _{1\le i<j\le n-1} d{u_{i,j}}\\
&=&\int_0^1\cdots \int_0^1 \int_{N_{n-1,1}(\bbZ)\backslash N_{n-1,1}(\bbR)} f\lp u' \cdot \bemfive 1&u_{1,2}&\cdots&u_{1,n-1}&0\\ &1&\cdots& u_{2,n-1}&0 \\ &&\ddots&\vdots&\vdots \\ &&&1&0\\ &&&&1\enm\cdot z\rp \prod _{1\le i\le n-1} d{u'_{i,n}}
\\ &&\times e(-m_1u_{1,2}-m_2u_{2,3} -\cdots - m_{n-2}u_{n-2,n-1})\prod _{1\le i<j\le n-1} d{u_{i,j}}\\
&=&0,
\eeq
by Proposition \ref{constant}.
For $\hat{f}_{(m_1,\ldots m_{n-4},0,0,m_{n-1})}$, the proof is similar as we have
\beq
u&=&\bemseven 1&0&\cdots&0&u_{1,n-2}&u_{1,n-1}&u_{1,n} -u_{1,n-1}u_{n-1,n} \\ 
&1&\ddots&\vdots&u_{2,n-2}&u_{2,n-1}&u_{2,n}  -u_{2,n-1}u_{n-1,n} \\
&&\ddots&0&\vdots&\vdots&\vdots\\
&& &1&u_{n-3,n-2}&u_{n-3,n-1}&u_{n-3,n} -u_{n-3,n-1}u_{n-1,n} \\
&& & &1&u_{n-2,n-1}&u_{n-2,n} -u_{n-2,n-1}u_{n-1,n}\\
&& & & &1&0\\
&& & & &&1 \enm\\
&&\times \bemseven 1&u_{1,2}&\cdots&u_{1,n-3}&0&0&0\\ &1&\cdots& u_{2,n-3}&0&0&0 \\ &&\ddots&\vdots&\vdots&\vdots&\vdots \\ &&&1&0&0&0\\ &&&&1&0&0\\ &&&&&1&u_{n-1,n}\\ &&&&&&1\enm.
\eeq
A simple change of variable $u_{i,n}-u_{i,n-1}u_{n-1,n}\mapsto u_{i,n}$ with the fact that $f$ is automorphic give the desired result again by Proposition \ref{constant}. Finally $a_{(0,\ldots,0)}=0$ is a direct result of the last part of Proposition \ref{constant}.
\end{proof}
 \vspace{20mm}
\subsection{Incomplete Eisenstein Series and Spectral Decomposition}

Eisenstein series induced form Maass forms play a central role in the spectral decomposition of the space $\frak{L}(\quo)$. But these Eisenstein series fail to be in $\frak{L}(\quo)$. So we need the theory of incomplete Eisenstein series. (They are  referred as pseudo-Eisenstein in \cite{MW}). Incomplete Eisenstein series were used in the study of quantum ergodicity of $GL(2)$ Eisenstein series in \cite{LS}. The analogous theory of incomplete Eisenstein series in more general settings can be found in \cite{L,MW,VE}. Here we give a brief summary for this theory on $GL(n)$ explicitly. First we clarify the definition of Mellin transform.
 
 \begin{dfn}
 For $\eta\in C_0^\infty ((\bbR^+)^{n})$, we define the $n$-dimensional Mellin transform of $\eta$ to be
 \beq
 \tilde{\eta}(s_1,\ldots,s_n) = \int_0^\infty\cdots\int_0^\infty \eta(y_1,\ldots,y_n) y_1^{-s_1}\cdots y_n^{-s_n}\ \frac{dy_1\cdots dy_{n}}{y_1\cdots y_n}.
 \eeq
 \end{dfn}

 \begin{dfn}\label{def4}
 
For $\eta\in C_0^\infty ((\bbR^+)^{r-1})$, we define the incomplete Eisenstein series $E_{(n_1,\ldots,n_r)}\lp z,\eta,\phi_1,\ldots,\phi_r\rp$ associated to $E_{(n_1,\ldots,n_r)}\lp z,s,\phi_1,\ldots,\phi_r\rp$ by
\begin{eqnarray*}
E_{(n_1,\ldots,n_r)}\lp z,\eta,\phi_1,\ldots,\phi_r\rp = \frac{1}{(2\pi i)^{r-1}}\int_{(2)}\cdots\int_{(2)} \tilde{\eta}(s_1,\ldots,s_{r-1}) E_{(n_1,\ldots,n_r)}\lp z,s,\phi_1,\ldots,\phi_r\rp  ds_1\cdots d s_{{r-1}}.
\end{eqnarray*}
 \end{dfn}
 Note that the convergence of these integrals in Definition \ref{def4} are guaranteed by the rapid decay of $\tilde{\eta}$ in imaginary parts of the arguments.

Let $\mathcal{H}_{cusp}$ denote the space of $GL(n)$ cusp forms, $\mathcal{H}_{(n_1,\ldots,n_r)}$ denote the space spanned by incomplete Eisenstein series $E_{(n_1,\ldots,n_r)}\lp z,\eta,\phi_1,\ldots,\phi_r\rp$. The following theorem is a summary of the results presented in Section II of \cite{MW} applied to the group $GL(n)$.
\begin{thm} \label{prop41}
We have the following spectral decomposition of automorphic forms on $GL(n)$:
\beq
\mathfrak{L}\lp\quo \rp = \mathcal{H}_{cusp} \oplus  \bigoplus_{((n_1,\ldots,n_r))}  \mathcal{H}_{(n_1,\ldots,n_r)}.
\eeq
Note that  $\mathcal{H}_{(n_1,\ldots,n_r)} =  \mathcal{H}_{(n'_1,\ldots,n'_r)}$ if $(n_1,\ldots,n_r)$ and $(n'_1,\ldots,n'_r)$ are associate. So $((n_1,\ldots,n_r))$ means that the decomposition is over representatives of associate partitions. For convenience, we always choose the partition in non-increasing order.
\end{thm}
From this decomposition, we will prove our main theorem by estimating the following integrals as $t\rightarrow\infty$ as mentioned in the introduction:
\beq
&&\int_{SL_n(\bbZ)\backslash X_n}\phi(z)\lbar E_{n-1,1}\lp z,\frac{1}{2}+it,1\rp\rbar^2d^*z,\\
&&\int_{SL_n(\bbZ)\backslash X_n}E_{1,\ldots,1}(z,\eta) \lbar E_{n-1,1}\lp z,\frac{1}{2}+it,1\rp\rbar^2d^*z,\\
&&\int_{SL_n(\bbZ)\backslash X_n}E_{n_1,\ldots,n_r}(z,\psi,u_1,\ldots,u_r) \lbar E_{n-1,1}\lp z,\frac{1}{2}+it,1\rp\rbar^2d^*z.
\eeq
where $\phi(z)$ is a $GL(n)$ cusp form.

 \newpage

\section{Fourier Expansion of Degenerate Eisenstein Series}
 
  \vspace{10mm}
\subsection{Fourier Expansion of Degenerate Eisenstein Series}
In this section, we present a complete calculation of the Fourier expansion of the degenerate Eisenstein series.

\begin{thm}\label{prop3}
The degenerate Eisenstein series for $GL(n)$ has only degenerate terms in its Fourier-Whittaker expansion for $n\ge3$. More precisely,
\beq
E_{(n-1,1)}(z,s,1) = \sum_{m_1\in\bbZ}\hat{\phi}_{(m_1,0,\ldots,0)}(z)+\sum_{i=2}^{n-1}\sum_{\gamma_i\in P_{(i-1,1)}(\bbZ)\backslash SL_i(\bbZ)} \sum_{m_i=1}^\infty \hat{\phi}_{(0,\ldots,0,m_i,0\ldots,0)}\lp\bemtwo \gamma_i &\\ & I_{n-i}\enm z\rp.
\eeq
where 
\beq
\hat{\phi}_{(m_1,\ldots,m_{n-1})}(z,s):=\int_0^1\cdots\int_0^1 E_{(n-1,1)}(u\cdot z,s,1) e(-m_1u_{1,2}-m_2u_{2,3}-\cdots-m_{n-1}u_{n-1,n})\prod _{1\le i<j\le n} d{u_{i,j}}
\eeq
with $u\in U_n(\bbR)$ and $I_{n-i}$ is the identity matrix of dimension $n-i$.

Furthermore, the coefficients are calculated to be
\beq
\hat{\phi}_{(0,\ldots,0)}(z,s) &=&   \sum_{k=0}^{n-1} \frac{2\xi(ns-n+k+1)}{\xi(ns)} \lp y_1 y_2^2\cdots y_{n-k-1}^{n-k-1}\rp^{(1-s)}    \lp y_{n-k}^{k}y_{n-(k-1)}^{k-1}\cdots y_{n-1}\rp^s; 
\eeq

\beq
\hat{\phi}_{(0,\ldots,0,m_k,0,\ldots,0)}(z,s) &=&   e(m_kx_{k,k+1})\frac{2}{\xi(ns)}|m_k|^{\frac{ns}{2}-\frac{n-k}{2}}\sigma_{-ns+n-k}(|m_k|) K_{\frac{ns}{2}-\frac{n-k}{2}}(2\pi |m_k|y_{n-k}) \\
&&\times \lp y_1 y_2^2\cdots y_{n-k-1}^{n-k-1}\rp^{(1-s)}    \lp y_{n-k}^{k}y_{n-(k-1)}^{k-1}\cdots y_{n-1}\rp^s   y_{n-k}^{-\frac{ns}{2}+\frac{n-k}{2}}.
\eeq

\end{thm}

We devote the reminding part of this section to proving the theorem. First rewrite the degenerate Eisenstein series in terms of Epstein Zeta function (see Section 10.7 of \cite{Go} for a proof):

\beq
\zeta(ns) E_{(n-1,1)}(z,s,1) = (y_1^{n-1}y_2^{n-2}\cdots y_{n-1})^s \sum_{(a_1,\ldots,a_n)\in \bbZ^n\backslash \{0\}} (b_1^2+\cdots+ b_n^2) ^{-ns/2}.
\eeq
For
\beq
z=  \bemfive 1 & x_{1,2}& x_{1,3}&\cdots& x_{1,n}\\  &1&x_{2,3}&\cdots & x_{2,n} \\ &&\ddots&&\vdots\\ &&&1& x_{n-1,n}\\&&&&1\enm   \bemfive   y_1y_2\cdots y_{n-1} &&&&\\  & y_1y_2\cdots y_{n-2}&&&\\ &&\ddots&&\\ &&&y_1&\\&&&&1\enm,
\eeq
the $b_i$'s are defined by
\beq
b_1&=& a_1y_1\cdots y_{n-1}\\
b_2&=&(a_1x_{1,2}+a_2) y_1\cdots y_{n-2}\\
&\vdots&\\
b_n&=&(a_1x_{1,n}+a_2x_{2,n}+\cdots+ a_{n-1}x_{n-1,n}+a_n).
\eeq
Let 
\beq
E_{(n-1,1)}^*(z,s) := \sum_{(a_1,\ldots,a_n)\in \bbZ^n\backslash \{0\}} (b_1^2+\cdots+ b_n^2) ^{-ns/2}.
\eeq
First, we separate the terms with $a_1=0$:
\beq
E_{(n-1,1)}^*(z,s) &=& \sum_{\substack{(a_2,\ldots,a_n)\in \bbZ^{n-1}\backslash\{0\} \\ a_1=0}}  (b_2^2+\cdots+ b_n^2) ^{-ns/2} + \sum_{a_1\neq 0} \sum_{(a_2,\ldots,a_n)\in \bbZ^{n-1}}  (b_1^2+\cdots+ b_n^2) ^{-ns/2}\\
&=&E_{(n-2,1)}^*\lp \pi_{n,n-1}(z),\frac{n}{n-1}s\rp + \mathcal{E}_n(z,s), \mbox{ say.}  
\eeq
Here $\pi_{n,n-k}:X_n\rightarrow X_{n-k}$ is the projection map:
\beq
\pi:\bemfive 1 & x_{1,2}& x_{1,3}&\cdots& x_{1,n}\\  &1&x_{2,3}&\cdots & x_{2,n} \\ &&\ddots&&\vdots\\ &&&1& x_{n-1,n}\\&&&&1\enm   \bemfive   y_1y_2\cdots y_{n-1} &&&&\\  & y_1y_2\cdots y_{n-2}&&&\\ &&\ddots&&\\ &&&y_1&\\&&&&1\enm\\
\mapsto \bemfive 1 & x_{1+k,3}& x_{1+k,4}&\cdots& x_{1+k,n}\\  &1&x_{2+k,4}&\cdots & x_{2+k,n} \\ &&\ddots&&\vdots\\ &&&1& x_{n-1,n}\\&&&&1\enm   \bemfive   y_1y_2\cdots y_{n-1-k} &&&&\\  & y_1y_2\cdots y_{n-2-k}&&&\\ &&\ddots&&\\ &&&y_1&\\&&&&1\enm.
\eeq
By Poisson summation,
\beq
\mathcal{E}_n(z,s)&=&\sum_{a_1\neq 0}\sum_{(c_2,\ldots,c_n)\in \bbZ^{n-1}}\int_{-\infty}^\infty\cdots \int_{-\infty}^\infty \frac{e\lp - c_2a_2-\cdots - c_na_n\rp}{(b_1^2+\cdots+ b_n^2)^{ns/2}} da_2\cdots da_n.
\eeq
Thus the Fourier-Whittaker coefficients of $E_{(n-1,1)}^*(z,s)$ are separated into two parts:
\begin{eqnarray}
\frac{\zeta(ns)}{(y_1^{n-1}y_2^{n-2}\cdots y_{n-1})^s}\hat{\phi}_{(m_1,\ldots,m_{n-1})}(z,s) &=& \int_0^1\cdots\int_0^1\lp    E_{(n-2,1)}^*\lp \pi_{n,n-1}(u\cdot z),\frac{n}{n-1}s\rp  +\mathcal{E}_n(u\cdot z,s)\rp\nonumber \\ 
&&\times e(-m_1u_{1,2}-m_2u_{2,3}-\cdots-m_{n-1}u_{n-1,n})\prod _{1\le i<j\le n} d{u_{i,j}} \nonumber\\
&=&\hat{\phi}_{(m_1,\ldots,m_{n-1})}^1(z,s)+\hat{\phi}_{(m_1,\ldots,m_{n-1})}^2(z,s),  \label{induction}
\end{eqnarray}
say. We treat these two parts in the following two lemmata.

\begin{lem}
The first parts of the coefficients $\hat{\phi}_{(m_1,\ldots,m_{n-1})}^1(z,s)$ have the following properties:  \label{lemma31}
\begin{enumerate}
\item $\hat{\phi}_{(m_1,\ldots,m_{n-1})}^1(z,s)=0$ if $m_1\neq 0$;
\item $\hat{\phi}_{(0,m_2,\ldots,m_{n-1})}^1(z,s) =  \frac{\zeta(ns)}{     \lp y_1^{n-2} y_2^{n-3}\cdots y_{n-2}\rp^{\frac{n}{n-1}s}      }\hat{\phi}_{(m_2,\ldots,m_{n-1})}\lp \pi_{n,n-1}(z), \frac{n}{n-1}s\rp.$
\end{enumerate}
\end{lem}

\begin{proof}
By direct matrix multiplication, we find that
\begin{equation}
\pi_{n,n-1}(u\cdot z) = \pi_{n,n-1}(u)\cdot \pi_{n,n-1}(z).\label{proj}
\end{equation}
Thus the integral against the variable $u_{1,2}$ gives $0$ unless $m_1=0$. This proves part (1).

For part (2), again by equation \ref{proj} the integrals against the variables $u_{1,j}$ for $2\le j\le n$ have no effect. So
\beq
 &&\int_0^1\cdots\int_0^1    E_{(n-2,1)}^*\lp \pi_{n,n-1}(u\cdot z),\frac{n}{n-1}s\rp  e(-m_1u_{1,2}-m_2u_{2,3}-\cdots-m_{n-1}u_{n-1,n})\prod _{1\le i<j\le n} d{u_{i,j}}\\
 &=&\int_0^1\cdots\int_0^1    E_{(n-2,1)}^*\lp \pi_{n,n-1}(u)\cdot \pi_{n,n-1}(z),\frac{n}{n-1}s\rp  e(-m_2u_{2,3}-\cdots-m_{n-1}u_{n-1,n})\prod _{2\le i<j\le n} d{u_{i,j}}\\
 &=& \frac{\zeta(ns)}{     \lp y_1^{n-2} y_2^{n-3}\cdots y_{n-2}\rp^{\frac{n}{n-1}s}      }\hat{\phi}_{(m_2,\ldots,m_{n-1})}\lp \pi_{n,n-1}(z), \frac{n}{n-1}s\rp.
\eeq
\end{proof}

\begin{lem}
The term $\hat{\phi}_{(m_1,\ldots,m_{n-1})}^2(z,s)$ in equation \ref{induction} can be nonzero only if $m_2=\cdots=m_{n-1}=0$. For $m_1\neq0$
\beq\hat{\phi}_{(m_1,0\ldots,0)}^2(z,s) &=& e(m_1x_{1,2})\frac{2\pi^{\frac{ns}{2}}}{\Gamma\lp\frac{ns}{2}\rp}|m_1|^{\frac{ns}{2}-\frac{n-1}{2}}\sigma_{-ns+n-1}(|m_1|)(y_1^{n-3}\cdots y_{n-3})^{-1} (y_1\cdots y_{n-2})^{-(ns-n+2)}\\
&&\times y_{n-1}^{-\frac{ns}{2}+\frac{n-1}{2}} K_{\frac{ns}{2}-\frac{n-1}{2}}(2\pi |m_1|y_{n-1});\eeq
for $m_1=0$
\beq
&&\hat{\phi}_{(0,\ldots,0)}^2(z,s) \\&=& \frac{2\pi^{\frac{n-1}{2}} \zeta(ns-n+1) \Gamma\lp\frac{ns}{2}-\frac{n-1}{2}\rp}{\Gamma\lp\frac{ns}{2}\rp} (y_1^{n-2}\cdots y_{n-2})^{-1} (y_1\cdots y_{n-1})^{-ns+n-1} .
\eeq
\end{lem}
\begin{proof}
Let $x'$ denote the product of matrices $u\cdot x$. Define $b_i'$ to be the same as $b_i$ with $x$ replaced by $x'$. We now make the change of variable $b_n'\mapsto a_n$:
\beq
\hat{\phi}_{(m_1,\ldots,m_{n-1})}^2(z,s)&=& \int_0^1\cdots\int_0^1 \sum_{a_1\neq 0}\sum_{(c_2,\ldots,c_n)\in \bbZ^{n-1}}\int_{-\infty}^\infty\cdots \int_{-\infty}^\infty \\
&&\frac{e\lp - c_2a_2-\cdots - c_n(a_n-a_1x_{1,n}'-\ldots-a_{n-1}x_{n-1,n}')\rp}{({b'}_1^2+\cdots+ a_n^2)^{ns/2}} da_2\cdots da_n\\
&&\times e(-m_1u_{1,2}-m_2u_{2,3}-\cdots-m_{n-1}u_{n-1,n}) \prod _{1\le i<j\le n} d{u_{i,j}}.
\eeq
We notice that the only entry with $u_{1,n}$ in the matrix $x'$ is the $x'_{1,n}$ entry. As $a_1\neq0$, the integral against $u_{1,n}$ gives $0$ unless $c_n=0$. After this simplification we have
\beq
\hat{\phi}_{(m_1,\ldots,m_{n-1})}^2(z,s)& = & \int_0^1\cdots\int_0^1  \sum_{a_1\neq 0}\sum_{(c_2,\ldots,c_{n-1})\in \bbZ^{n-2}} \int_{-\infty}^\infty\cdots \int_{-\infty}^\infty \frac{e\lp - c_2a_2-\cdots - c_{n-1}a_{n-1}\rp}{({b'}_1^2+\cdots+ a_n^2)^{ns/2}} da_2\cdots da_n   \\ &&\times e(-m_1u_{1,2}-m_2u_{2,3}-\cdots-m_{n-1}u_{n-1,n})\frac{\prod _{1\le i<j\le n} d{u_{i,j}}}{du_{1,n}}
\eeq
Now make the change of variable $ \frac{ b'_{n-1}}{y_1}\mapsto a_{n-1}$:
\beq
\hat{\phi}_{(m_1,\ldots,m_{n-1})}^2(z,s) & = & \int_0^1\cdots\int_0^1  \sum_{a_1\neq 0}\sum_{(c_2,\ldots,c_{n-1})\in \bbZ^{n-2}} \int_{-\infty}^\infty\cdots \int_{-\infty}^\infty  \\ &&\times\frac{e\lp - c_2a_2-\cdots - c_{n-1}(a_{n-1}-a_1{x'}_{1,n-1}-\cdots- a_{n-2}{x'}_{n-2,n-1})\rp}{({b'}_1^2+\cdots+a_{n-1}^2y_1^2+ a_n^2)^{ns/2}} da_2\cdots da_n \\&&\times e(-m_1u_{1,2}-m_2u_{2,3}-\cdots-m_{n-1}u_{n-1,n})\frac{\prod _{1\le i<j\le n} d{u_{i,j}}}{du_{1,n}}
\eeq
The only term with $u_{1,n-1}$ is in the entry ${x'}_{1,n-1}$. Integrating against $u_{1,n-1}$ with the fact that $a_1\neq0$ shows $c_{n-1}=0$. Continuing in this fashion, we get
\begin{eqnarray}
\label{poisson}\hat{\phi}_{(m_1,\ldots,m_{n-1})}^2(z,s) & = & \int_0^1\cdots\int_0^1  \sum_{a_1\neq 0}\sum_{c_2\in \bbZ} \int_{-\infty}^\infty\cdots \int_{-\infty}^\infty \frac{e\lp - c_2(a_2-a_1{x}_{1,2}-a_1 u_{1,2})\rp}{({a}_1^2y_1^2\cdots y_{n-1}^2+\cdots+ a_n^2)^{ns/2}} da_2\cdots da_n   \\ &&\times  e(-m_1u_{1,2}-m_2u_{2,3}-\cdots-m_{n-1}u_{n-1,n})\frac{\prod _{1\le i<j\le n} d{u_{i,j}}}{du_{1,n}\cdots du_{1,3}}. \nonumber
\end{eqnarray}
Clearly, this is zero unless $m_2=\cdots =m_{n-1}=0$.

To continue, we will use the following well-known identities:
\begin{eqnarray}
\int_{-\infty}^\infty (A^2x^2+C)^{-\nu} dx&=&\sqrt{\pi}\frac{\Gamma\lp\nu-\frac{1}{2}\rp}{\Gamma(\nu)}|A|^{-1}C^{-\nu+\frac{1}{2}}\label{eqn1},\\
\int_{-\infty}^\infty (A^2x^2+C^2)^{-\nu}e(-Dx) dx& =& \frac{2\pi^\nu}{\Gamma({\nu})}|A|^{-\nu-\frac{1}{2}}|D|^{\nu-\frac{1}{2}}|C|^{\frac{1}{2}-\nu} K_{\nu-\frac{1}{2}}\lp\frac{2\pi |CD|}{A}\rp,\label{eqn2}
\end{eqnarray}

where $A,C,D\in\bbR$. Continuing from \ref{poisson}, we have
\beq
\hat{\phi}_{(m_1,0\ldots,0)}^2(z,s) &=& \int_0^1  \sum_{a_1\neq 0}\sum_{c_2\in \bbZ} \int_{-\infty}^\infty\cdots \int_{-\infty}^\infty \frac{e\lp - c_2(a_2-a_1{x}_{1,2}-a_1 u_{1,2})\rp}{({a}_1^2y_1^2\cdots y_{n-1}^2+\cdots+ a_n^2)^{ns/2}} da_2\cdots da_n   \  e(-m_1u_{1,2}) du_{1,2}.
\eeq

For $m_1=0$,
\beq
\hat{\phi}_{(0,0\ldots,0)}^2(z,s) &=&   \int_0^1  \sum_{a_1\neq 0}\sum_{c_2\in \bbZ} \int_{-\infty}^\infty\cdots \int_{-\infty}^\infty \frac{e\lp - c_2(a_2-a_1{x}_{1,2}-a_1 u_{1,2})\rp}{({a}_1^2y_1^2\cdots y_{n-1}^2+\cdots+ a_n^2)^{ns/2}} da_2\cdots da_n\ du_{1,2}\\
&=&\sum_{a_1\neq0}\int_{-\infty}^\infty\cdots \int_{-\infty}^\infty \frac{1}{({a}_1^2y_1^2\cdots y_{n-1}^2+\cdots+ a_n^2)^{ns/2}} da_2\cdots da_n\\
&=&\frac{2\pi^{\frac{n-1}{2}} \zeta(ns-n+1) \Gamma\lp\frac{ns}{2}-\frac{n-1}{2}\rp}{\Gamma\lp\frac{ns}{2}\rp} (y_1^{n-2}\cdots y_{n-2})^{-1} (y_1\cdots y_{n-1})^{-ns+n-1} 
\eeq
after repeatedly applying \ref{eqn1}. For $m_1\neq 0$,
\beq
\hat{\phi}_{(m_1,0\ldots,0)}^2(z,s) &=& \int_0^1  \sum_{a_1\neq 0}\sum_{c_2\in \bbZ} \int_{-\infty}^\infty\cdots \int_{-\infty}^\infty \frac{e\lp - c_2(a_2-a_1{x}_{1,2}-a_1 u_{1,2})\rp}{({a}_1^2y_1^2\cdots y_{n-1}^2+\cdots+ a_n^2)^{ns/2}} da_2\cdots da_n   \  e(-m_1u_{1,2}) du_{1,2}\\
&=&e(m_1 x_{1,2})\doublesum_{\substack{a_1\neq 0 c_2\in \bbZ \\a_1c_2=m_1}} \int_{-\infty}^\infty\cdots \int_{-\infty}^\infty   \frac{e\lp - c_2a_2\rp}{({a}_1^2y_1^2\cdots y_{n-1}^2+\cdots+ a_n^2)^{ns/2}} da_2\cdots da_n\\
&=& e(m_1x_{1,2})\frac{2\pi^{\frac{ns}{2}}}{\Gamma\lp\frac{ns}{2}\rp}|m_1|^{\frac{ns}{2}-\frac{n-1}{2}}\sigma_{-ns+n-1}(|m_1|)(y_1^{n-3}\cdots y_{n-3})^{-1} (y_1\cdots y_{n-2})^{-(ns-n+2)}\\
&&\times y_{n-1}^{-\frac{ns}{2}+\frac{n-1}{2}} K_{\frac{ns}{2}-\frac{n-1}{2}}(2\pi |m_1|y_{n-1}).
\eeq
\end{proof}

Finally the calculation of the Fourier expansion is complete after a simple induction on $n$ using equation \ref{induction}.

\vspace{20mm}

\subsection{Constant Term Computation}
In this section, we utilize the Fourier expansion computation to investigate the following two constant terms:
\beq
\int_{(\bbZ\backslash \bbR)^{n-1}}    E_{(n-1,1)}\lp z,s,1   \rp \  \prod_{h=1}^{n-1} dx_{h,n},
\eeq
\beq
\int_{(\bbZ\backslash \bbR)^{n-1}}     \lbar   E_{(n-1,1)}\lp z,\frac{1}{2}+it,1\rp     \rbar^2 \  \prod_{h=1}^{n-1} dx_{h,n}
\eeq
to prepare for calculation in next section. First, we need a lemma.

\begin{lem} \label{action}
Let $\gamma_k = \bemfive &&&&\\&&*&&\\&&&&\\ \gamma_{k,1}&\gamma_{k,2}&\cdots&\gamma_{k,k-1}&\gamma_{k,k}\enm\in SL_k(\bbZ)$ and let 
\beq
z'= \bemfive 1 & x'_{1,2}& x'_{1,3}&\cdots& x'_{1,n}\\  &1&x'_{2,3}&\cdots & x'_{2,n} \\ &&\ddots&&\vdots\\ &&&1& x'_{n-1,n}\\&&&&1\enm\cdot \bemfive   y'_1y'_2\cdots y'_{n-1} &&&&\\  & y_1y_2\cdots y_{n-2}&&&\\ &&\ddots&&\\ &&&y_1&\\&&&&1\enm
\eeq
be the product $\bemtwo \gamma_k& \\ & I_{n-k}\enm \cdot z$ in Iwasawa form, then
\begin{enumerate}
\item 
\beq
y_j'=y_j\hspace{10mm} \mbox{ for }1\le j\le n-k-1;
\eeq
\item 
\beq
{y'}_{n-k}^{k}{y'}_{n-(k-1)}^{k-1}\cdots y'_{n-1} = {y}_{n-k}^{k}{y}_{n-(k-1)}^{k-1}\cdots y_{n-1};
\eeq
\item
\beq
y'_{n-k} = y_{n-k}\lp b_1^2+\cdots + b_{k}^2\rp^{\frac{1}{2}}
\eeq
where
\beq
b_1&=& \gamma_{k,1} y_{n-k+1}\cdots y_{n-1}\\
b_2 &= &(\gamma_{k,1}x_{1,2}+\gamma_{k,2})y_{n-k+1}\cdots y_{n-2}\\
&\vdots&\\
b_{k-1} &=&(\gamma_{k,1}x_{1,k-1}+\gamma_{k,2}x_{2,k-1}+\cdots \gamma_{k,k-2}x_{k-2,k-1}+\gamma_{k,k-1})y_{n-k+1}\\
b_{k} &=&\gamma_{k,1}x_{1,k}+\gamma_{k,2}x_{2,k}+\cdots \gamma_{k,k-1}x_{k-1,k}+\gamma_{k,k}.
\eeq
\item \beq
x'_{k,k+1} = \sum_{i=1}^k \gamma_{k,i} x_{i,k+1}.
\eeq
\end{enumerate}
\end{lem}
\begin{proof}
First we write
\beq
x = \bemtwo X_1&V\\ &X_2\enm && y= \bemtwo Y_1 y_1\cdots y_{n-k-1}&\\& Y_2\enm=\bemtwo \tilde{Y_1} y_1\cdots y_{n-k}&\\& Y_2\enm,
\eeq
where $x\cdot y$ is the Iwasawa decomposition of $z$ and $X_1,Y_1$ are of dimension $k\times k$, $X_2,Y_2$ are of dimension $(n-k)\times(n-k)$.
Matrix multiplication gives
\beq
\bemtwo \gamma_k&\\&I_{n-k}\enm  \bemtwo X_1&V\\ &X_2\enm \bemtwo Y_1 y_1\cdots y_{n-k-1}&\\& Y_2\enm =  \bemtwo\gamma_k \cdot X_1 \cdot Y_1&\gamma_k \cdot V \\&X_2\enm \bemtwo  I_k y_1\cdots y_{n-k-1}&\\ Y_2\enm.
\eeq
By the Iwasawa decomposition, the matrix $\gamma_i\cdot X_1\cdot Y_1$ can be written as $X_1'\cdot Y_1'\cdot K_1'$. So
\beq
\bemtwo \gamma_k&\\&I_{n-k}\enm z =  \bemtwo X_1'& \gamma_k\cdot V\\& X_2 \enm    \bemtwo Y_1' y_1\cdots y_{n-k-1}&\\& Y_2\enm  \bemtwo K'&  \\&1\enm.
\eeq
We have put $\bemtwo \gamma_k&\\&I_{n-k}\enm z $ in Iwasawa form. Note that $\det(Y_1)=\det(Y_1') = y_{n-k}^{k}y_{n-(k-1)}^{k-1}\cdots y_{n-1}$ is invariant under the action of $\gamma_k$ as $\gamma_k\in SL_k(\bbZ)$ has determinant $1$. At the same time, the variables $y_1,y_2,\ldots,y_{n-k-1}$ are unchanged.

The proof of the final part of this lemma is presented on P.309 of \cite{Go}. 
\end{proof}

\begin{prop}\label{constantterm}
\begin{eqnarray}
\nonumber&&\int_{(\bbZ\backslash \bbR)^{n-1}}    E_{(n-1,1)}\lp z,s,1   \rp \  \prod_{h=1}^{n-1} dx_{h,n}\\
&=&2\lp y_1^{n-1}\cdots y_{n-1}\rp^s + \frac{\xi(ns-1)}{\xi(ns)} \lp y_1^{n-1}\cdots y_{n-1}\rp^{\frac{1-s}{n-1}}E_{(n-2,1)}\lp \frak{m}_{n-1}(z),\frac{ns-1}{n-1},1   \rp.
\end{eqnarray}
\end{prop}
\begin{proof}
By Theorem \ref{prop3}
\beq
&&\int_{(\bbZ\backslash \bbR)^{n-1}}    E_{(n-1,1)}\lp z,s,1   \rp \  \prod_{h=1}^{n-1} dx_{h,n}\\
&=&\int_{(\bbZ\backslash \bbR)^{n-1}} \lp  \sum_{m_1\in\bbZ}\hat{\phi}_{(m_1,0,\ldots,0)}(z)+\sum_{i=2}^{n-1}\sum_{\gamma_i\in P_{(i-1,1)}(\bbZ)\backslash SL_i(\bbZ)} \sum_{m_i=1}^\infty \hat{\phi}_{(0,\ldots,0,m_i,0\ldots,0)}\lp\bemtwo \gamma_i &\\ & I_{n-i}\enm z\rp \rp \  \prod_{h=1}^{n-1} dx_{h,n}
\eeq
By the previous lemma, we see that only the term 
\beq
\sum_{\gamma_i\in P_{(i-1,1)}(\bbZ)\backslash SL_i(\bbZ)} \sum_{m_i=1}^\infty \hat{\phi}_{(0,\ldots,0,m_i,0\ldots,0)}\lp\bemtwo \gamma_i &\\ & I_{n-i}\enm z\rp
\eeq
has the variables $x_{h,n}$ for $1\le h\le n-1$ and these variables only appears in the exponential factor as $e\lp \sum_{j=1}^{n-1}\gamma_{n-1,j}x_{j,n}\rp$, where $\gamma_{a,b}$ are entries of the matrix $\gamma_{n-1}$.
Thus
\beq
\int_{(\bbZ\backslash \bbR)^{n-1}}   \sum_{\gamma_i\in P_{(i-1,1)}(\bbZ)\backslash SL_i(\bbZ)} \sum_{m_i=1}^\infty \hat{\phi}_{(0,\ldots,0,m_i,0\ldots,0)}\lp\bemtwo \gamma_i &\\ & I_{n-i}\enm z\rp  \  \prod_{h=1}^{n-1} dx_{h,n}=0.
\eeq
So we have
\beq
&&\int_{(\bbZ\backslash \bbR)^{n-1}}    E_{(n-1,1)}\lp z,s,1   \rp \  \prod_{h=1}^{n-1} dx_{h,n}\\
&=& \sum_{m_1\in\bbZ}\hat{\phi}_{(m_1,0,\ldots,0)}(z)+\sum_{i=2}^{n-2}\sum_{\gamma_i\in P_{(i-1,1)}(\bbZ)\backslash SL_i(\bbZ)} \sum_{m_i=1}^\infty \hat{\phi}_{(0,\ldots,0,m_i,0\ldots,0)}\lp\bemtwo \gamma_i &\\ & I_{n-i}\enm z\rp\\
&=&2\lp y_1^{n-1}\cdots y_{n-1}\rp^s + \frac{\xi(ns-1)}{\xi(ns)} \lp y_1^{n-1}\cdots y_{n-1}\rp^{\frac{1-s}{n-1}}E_{(n-2,1)}\lp \frak{m}_{n-1}(z),\frac{ns-1}{n-1},1   \rp
\eeq
by comparing Fourier expansions. 
\end{proof}

\begin{prop}\label{constantsquare}
\begin{eqnarray*}&&\int_{(\bbZ\backslash \bbR)^{n-1}}     \lbar   E_{(n-1,1)}\lp z,s,1\rp     \rbar^2 \  \prod_{h=1}^{n-1} dx_{h,n}\\
\nonumber&=& 4 \ell^{-2na} +\frac{\lbar \xi(ns-1) \rbar^2}{\lbar \xi(ns) \rbar^2}\lp  y_1^{n-1}\cdots y_{n-1}\rp^{\frac{2-a}{n-1}}  \lbar E_{(n-2,1)}\lp \frak{m}_{n-1}(z),\frac{ns-1}{n-1} ,1   \rp\rbar^2  \\
\nonumber&&+2\ell^{-\frac{n^2a-2na+n+in^2b}{n-1}}\frac{\xi(n\bar{s}-1)}{\xi(n\bar{s})} E_{(n-2,1)}\lp \frak{m}_{n-1}(z),\frac{n\bar{s}-1}{n-1},1   \rp\\
\nonumber&&+2\ell^{-\frac{n^2a-2na+n-in^2b}{n-1}}\frac{\xi(ns-1)}{\xi(ns)} E_{(n-2,1)}\lp \frak{m}_{n-1}(z),\frac{ns-1}{n-1},1   \rp\\
\nonumber&&+\frac{8\ell^{-2na}}{\lbar\xi \lp s\rp\rbar^2}\sum_{m_{n-1}=1}^\infty\sum_{\gamma\in P_{(n-2,1)}(\bbZ)\backslash SL_{n-1}(\bbZ)} m_{n-1}^{na-1}\sigma_{-na+1-inb}(m_{n-1})\sigma_{-na+1+inb}(m_{n-1}) \lp\frac{\ell^{-\frac{n}{n-1}}}{\det(\frak{m}_{n-1}(z))^{\frac{1}{n-1}}}\rp^{-na+1}\\
\nonumber&&\times \ \left.K_{\frac{na-1}{2}+\frac{inb}{2}}\lp2\pi m_{n-1}\frac{\ell^{-\frac{n}{n-1}}}{\det(\frak{m}_{n-1}(z))^{\frac{1}{n-1}}}\rp K_{\frac{na-1}{2} -\frac{inb}{2}}\lp2\pi m_{n-1}\frac{\ell^{-\frac{n}{n-1}}}{\det(\frak{m}_{n-1}(z))^{\frac{1}{n-1}}}\rp \right|_\gamma
\end{eqnarray*}
where $s=a+bi$, $\ell^{-n} = y_1^{n-1}\cdots y_{n-1}$ and $\det(\frak{m}_{n-1}(z)) = y_2^{n-2}\cdots y_{n-1}$.
\end{prop}
\begin{proof}
By the previous lemma we can open the square and find

\beq
&&\int_{(\bbZ\backslash \bbR)^{n-1}}     \lbar   E_{(n-1,1)}\lp z,s,1\rp     \rbar^2 \  \prod_{h=1}^{n-1} dx_{h,n}\\
&=&  \lbar 2\lp y_1^{n-1}\cdots y_{n-1}\rp^s + \frac{\xi(ns-1)}{\xi(ns)} \lp y_1^{n-1}\cdots y_{n-1}\rp^{\frac{1-s}{n-1}}E_{(n-2,1)}\lp \frak{m}_{n-1}(z),\frac{ns-1}{n-1},1   \rp\rbar^2\\
&&+\frac{8\lp  y_1^{n-1}\cdots y_{n-1}\rp^{2a}}{\lbar\xi \lp s\rp\rbar^2}\sum_{m_{n-1}=1}^\infty\sum_{\gamma\in P_{(n-2,1)}(\bbZ)\backslash SL_{n-1}(\bbZ)} m_{n-1}^{na-1}\sigma_{-na+1-inb}(m_{n-1})\sigma_{-na+1+inb}(m_{n-1})\\
&&\times \ \left.K_{\frac{na-1}{2}+\frac{inb}{2}}(2\pi m_{n-1}y_1)K_{\frac{na-1}{2} -\frac{inb}{2}}(2\pi m_{n-1}y_1) y_1^{-na+1}\right|_\gamma\\
&=&4\lp  y_1^{n-1}\cdots y_{n-1}\rp^{2a} +\frac{\lbar \xi(ns-1) \rbar^2}{\lbar \xi(ns) \rbar^2}\lp  y_1^{n-1}\cdots y_{n-1}\rp^{\frac{2-a}{n-1}}  \lbar E_{(n-2,1)}\lp \frak{m}_{n-1}(z),\frac{ns-1}{n-1} ,1   \rp\rbar^2  \\
&&+2\lp  y_1^{n-1}\cdots y_{n-1}\rp^{\frac{na-2a+1+inb}{n-1}}\frac{\xi(n\bar{s}-1)}{\xi(n\bar{s})} E_{(n-2,1)}\lp \frak{m}_{n-1}(z),\frac{n\bar{s}-1}{n-1},1   \rp\\
&&+2\lp  y_1^{n-1}\cdots y_{n-1}\rp^{\frac{na-2a+1-inb}{n-1}}\frac{\xi(ns-1)}{\xi(ns)} E_{(n-2,1)}\lp \frak{m}_{n-1}(z),\frac{ns-1}{n-1},1   \rp\\
&&+\frac{8\lp  y_1^{n-1}\cdots y_{n-1}\rp^{2a}}{\lbar\xi \lp s\rp\rbar^2}\sum_{m_{n-1}=1}^\infty\sum_{\gamma\in P_{(n-2,1)}(\bbZ)\backslash SL_{n-1}(\bbZ)} m_{n-1}^{na-1}\sigma_{-na+1-inb}(m_{n-1})\sigma_{-na+1+inb}(m_{n-1})\\
&&\times \ \left.K_{\frac{na-1}{2}+\frac{inb}{2}}(2\pi m_{n-1}y_1)K_{\frac{na-1}{2} -\frac{inb}{2}}(2\pi m_{n-1}y_1) y_1^{-na+1}\right|_\gamma
\eeq

We get our desired result after the change of variable $\ell^{-n} = y_1^{n-1}\cdots y_{n-1}$ and $\det(\frak{m}_{n-1}(z)) = y_2^{n-2}\cdots y_{n-1}$.

\end{proof}

\newpage
\section{Cuspidal and Non-Minimal Eisenstein Contribution}
In this section we show that 
\begin{enumerate}
\item
\beq
\int_{SL_n(\bbZ)\backslash X_n}\phi(z)\lbar E_{n-1,1}\lp z,\frac{1}{2}+it,1\rp\rbar^2d^*z = 0,
\eeq
for $\phi(z)$ a cusp form on $GL(n)$ with $n\ge3$.
\item

\beq
\int_{SL_n(\bbZ)\backslash X_n}E_{(n_1,\ldots,n_r)}(z,\psi,u_1,\ldots,u_r) \lbar E_{n-1,1}\lp z,\frac{1}{2}+it,1\rp\rbar^2d^*z = 0
\eeq
if $n\ge3$ and the partition $(n_1,\ldots,n_r)$ is not of type $(2,1,\ldots,1)$.
\item
\beq
\int_{SL_n(\bbZ)\backslash X_n}E_{(2,1,\ldots,1)}(z,\psi,\phi) \lbar E_{n-1,1}\lp z,\frac{1}{2}+it,1\rp\rbar^2d^*z = O_\epsilon(t^{-\frac{1}{2}+\epsilon}).
\eeq
\end{enumerate}

\vspace{10mm}
\subsection{Some Basic Lemmata}

Here we write down some basic lemmata that we will use repeatedly in the rest of the paper.

\begin{lem}(Stade's formula for $GL(2)$)

The Mellin transform of product of $K$-Bessel functions can be evaluated (6.576.4 of \cite{GR})
 \begin{eqnarray}\label{stade}
 \int_0^\infty K_\mu(y)K_\nu(y)y^s\frac{dy}{y} = 2^{s-3}\frac{\Gamma\lp\frac{s+\mu+\nu}{2}\rp\Gamma\lp\frac{s+\mu-\nu}{2}\rp\Gamma\lp\frac{s-\mu+\nu}{2}\rp\Gamma\lp\frac{s-\mu-\nu}{2}\rp}{\Gamma(s)}.
 \end{eqnarray}
\end{lem}

\begin{lem}\label{divisor} We have the following beautiful formula given by Ramanujan \cite{Ra}.
 \begin{eqnarray}
\sum_{n=1}\frac{\sigma_a(n)\sigma_b(n)}{n^s}= \frac{\zeta(s)\zeta(s-a)\zeta(s-b)\zeta(s-a-b)}{\zeta(2s-a-b)}
 \end{eqnarray}
 for $\re(s)$ large enough.
\end{lem}
\begin{lem}(Mellin inversion) (See 3.1.1 of \cite{PK})

The Mellin transform of a absolutely integrable function $f(x)$ on $(0,\infty)$ is defined by 
\beq
F(s) = \int_0^\infty x^{s-1}f(x) dx
\eeq
when the integral converges. If we suppose for small $\epsilon>0$
\beq
f(x) = \lbra\begin{array}{lll} O(x^{-a-\epsilon})&\mbox{as}&x\rightarrow0^+\\ O(x^{-b+\epsilon})&\mbox{as}&x\rightarrow\infty, \end{array}\right.
\eeq
then for $a<c<b$, $f(x)$ can be recovered by Mellin inversion:
\beq
f(x) = \frac{1}{2\pi i}\int_{(c)} x^{-s}F(s) ds.
\eeq
\end{lem}

\vspace{20mm}
\subsection{Partition of the Form \texorpdfstring{$n=n_1+\cdots+n_{r-1}+1$}{Lg}} 
We now fix a partition $n=n_1+\cdots+n_{r-1}+1\neq 1+\cdots+1$ in non-increasing order. Our goal this section is to reduce the calculation involving Eisenstein series associated to the parabolic subgroup $P_{n_1,\ldots,n_{r-1},1}$ to the calculation involving Eisenstein series associated to the parabolic subgroup $P_{n_1,\ldots,n_{r-1}}$. This argument can be repeated so that we are left with partitions of the form $n=n_1+\cdots+n_r$ with $n_r\ge2$, which will be tackled in the next section.

First we have the following rearrangement of Eisenstein series
\begin{lem}
\begin{eqnarray}
&&E_{(n_1,\ldots,n_{r-1},1)}\lp z,s,\phi_1,\ldots,\phi_{r-1},1\rp\\
\nonumber&=&E_{(n-1,1)}\lp z,s'', E_{(n_1,\ldots,n_{r-1})}\lp \frak{m}_{n-1}(z),s',\phi_1,\ldots,\phi_{r-1}\rp ,1\rp
\end{eqnarray}
where $s' = (s'_1,\ldots,s'_{r-1})\in\bbC^{r-1}$ and $s''\in\bbC$ are given by
\beq
s'_i = s_i-\frac{n_1s_1+\cdots+n_{r-1}s_{r-1}}{n-1},&& s'' = \frac{n_1s_1+\cdots+n_{r-1}s_{r-1}}{n-1}.
\eeq
\end{lem}
\begin{proof}
By definition, 
\beq
&&E_{(n_1,\ldots,n_{r-1},1)}\lp z,s,\phi_1,\ldots,\phi_{r-1},1\rp\\
&=& \sum_{\gamma\in P_{n_1,\ldots,n_{r-1},1}(\bbZ)\backslash SL_n(\bbZ)} \prod_{i=1}^{r-1}\phi_i(\frak{m}_i(\gamma z))I_s(\gamma z,P_{n_1,\ldots,n_{r-1},1})\\
&=&\sum_{\beta \in P_{n_1,\ldots,n_{r-1},1}(\bbZ)\backslash P_{n-1,1}(\bbZ)}\   \sum_{\alpha\in P_{n-1,1}(\bbZ)\backslash SL_n(\bbZ)} \prod_{i=1}^{r-1}\phi_i(\frak{m}_i(\beta\alpha z))I_s(\beta\alpha z,P_{n_1,\ldots,n_{r-1},1})\\
\eeq
Here we can write
\beq
I_s\lp z,P_{n_1,\ldots,n_{r-1},1}\rp = I_{s'}\lp \frak{m}_{n-1}(z),P_{n_1,\ldots,n_{r-1}}\rp \ \lp y_1^{n-1} y_2^{n-2}\cdots y_{n-1}\rp ^{s''}.
\eeq
where $s' = (s'_1,\ldots,s'_{r-1})\in\bbC^{r-1}$ and $s''\in\bbC$ are given by
\beq
s'_i = s_i-\frac{n_1s_1+\cdots+n_{r-1}s_{r-1}}{n-1},&& s'' = \frac{n_1s_1+\cdots+n_{r-1}s_{r-1}}{n-1}.
\eeq
One easily can check that $n_1s'_1+\cdots+n_{r-1}s'_{r-1}=0$ so that $I_{s'}\lp \frak{m}_{n-1}(z),P_{n_1,\ldots,n_{r-1}}\rp$ is a well defined $I$-function with spectral parameter $s'$. Thus we have
\beq
&&E_{(n_1,\ldots,n_{r-1},1)}\lp z,s,\phi_1,\ldots,\phi_{r-1},1\rp\\
&=&\sum_{\beta \in P_{n_1,\ldots,n_{r-1},1}(\bbZ)\backslash P_{n-1,1}(\bbZ)}\   \sum_{\alpha\in P_{n-1,1}(\bbZ)\backslash SL_n(\bbZ)}\left.\left. \prod_{i=1}^{r-1}\phi_i(\frak{m}_i(  z))I_{s'}\lp \frak{m}_{n-1}(z),P_{n_1,\ldots,n_{r-1}}\rp   \lp y_1^{n-1} y_2^{n-2}\cdots y_{n-1}\rp ^{s''}    \rbar_\beta\rbar_\alpha\\
&=& \sum_{\alpha\in P_{n-1,1}(\bbZ)\backslash SL_n(\bbZ)} \left. \lp \sum_{\beta \in P_{n_1,\ldots,n_{r-1},1}(\bbZ)\backslash P_{n-1,1}(\bbZ)}  \prod_{i=1}^{r-1}\phi_i(\frak{m}_i(  \beta z))I_{s'}\lp \frak{m}_{n-1}(\beta z),P_{n_1,\ldots,n_{r-1}}\rp  \rp  \lp y_1^{n-1} y_2^{n-2}\cdots y_{n-1}\rp ^{s''}\rbar_\alpha\\
&=& \sum_{\alpha\in P_{n-1,1}(\bbZ)\backslash SL_n(\bbZ)} \left. \lp \sum_{\beta \in P_{n_1,\ldots,n_{r-1}}(\bbZ)\backslash SL_{n-1}(\bbZ)}  \prod_{i=1}^{r-1}\phi_i(\frak{m}_i(  \beta z))I_{s'}\lp \frak{m}_{n-1}(\beta z),P_{n_1,\ldots,n_{r-1}}\rp  \rp  \lp y_1^{n-1} y_2^{n-2}\cdots y_{n-1}\rp ^{s''}\rbar_\alpha\\
&=&\left. \sum_{\alpha\in P_{n-1,1}(\bbZ)\backslash SL_n(\bbZ)} E_{(n_1,\ldots,n_{r-1})}\lp \frak{m}_{n-1}(z),s',\phi_1,\ldots,\phi_{r-1}\rp  \lp y_1^{n-1} y_2^{n-2}\cdots y_{n-1}\rp ^{s''}\rbar_\alpha\\
&=&E_{(n-1,1)}\lp z,s'', E_{(n_1,\ldots,n_{r-1})}\lp \frak{m}_{n-1}(z),s',\phi_1,\ldots,\phi_{r-1}\rp ,1\rp
\eeq
as $\lp y_1^{n-1} y_2^{n-2}\cdots y_{n-1}\rp$ is invariant under the action of $\beta$ by Lemma \ref{action}.
\end{proof}
\begin{lem}\label{changeof}
We have the decomposition of measurable space
\begin{eqnarray}
P_{n-1,1}(\bbZ)\backslash X_n = SL_{n-1}(\bbZ)\backslash X_{n-1}\times (\bbR/\bbZ)^{n-1}\times\bbR^+.
\end{eqnarray}
with
\beq
d^*z = -\frac{1}{\xi(n)}\ d^*z'  \  \prod_{h=1}^{n-1} dx_{h,n} \   \ell^n\ \frac{d \ell}{\ell}.
\eeq
where
\beq
\ell  = \lp\prod_{i=1}^{n-1}y_i^{n-i}\rp^{-\frac{1}{n}}&& z' = \frak{m}_{n-1}(z).
\eeq
\end{lem}
\begin{proof}
By matrix multiplication, we have
\beq
z=X\cdot Y&=&\bemfive 1 & x_{1,2}& x_{1,3}&\cdots& x_{1,n}\\  &1&x_{2,3}&\cdots & x_{2,n} \\ &&\ddots&&\vdots\\ &&&1& x_{n-1,n}\\&&&&1\enm   \bemfive   \ell\ y_1y_2\cdots y_{n-1} &&&&\\  & \ell \ y_1y_2\cdots y_{n-2}&&&\\ &&\ddots&&\\ &&&\ell \ y_1&\\&&&&\ell\enm\\
&=&\bemfive 1 & &  & & x_{1,n}\\  &1& &  & x_{2,n} \\ &&\ddots&&\vdots\\ &&&1& x_{n-1,n}\\&&&&1\enm  \cdot \bemtwo z'&\\ &1\enm\cdot \bemtwo \ell^{-\frac{1}{n-1}}\cdot I_{n-1}&\\ &\ell\enm\\
&=& \bemtwo I_{n-1}& \overline{x}_n\\ &1\enm\cdot \bemtwo z'&\\ &1\enm\cdot \bemtwo \ell^{-\frac{1}{n-1}}\cdot I_{n-1}&\\ &\ell\enm,
\eeq
where
\beq
z'= \bemfive 1 & x_{1,2}& x_{1,3}&\cdots& x_{1,n-1}\\  &1&x_{2,3}&\cdots & x_{2,n-1} \\ &&\ddots&&\vdots\\ &&&1& x_{n-2,n-1}\\&&&&1\enm   \bemfive   \ell^{\frac{n}{n-1}}\ y_1y_2\cdots y_{n-1} &&&&\\  & \ell^{\frac{n}{n-1}} \ y_1y_2\cdots y_{n-2}&&&\\ &&\ddots&&\\ &&&\ell^{\frac{n}{n-1}} \ y_1&\enm = X'\cdot Y'.
\eeq

By direct computation, we have for $\gamma\in SL_{n-1}(\bbZ)$
\begin{eqnarray*}
&&\bemtwo    \gamma&\\ &1\enm \cdot z\\
&=&\bemtwo    \gamma&\\ &1\enm  \cdot \bemtwo I_{n-1}& \overline{x}_n\\ &1\enm\cdot \bemtwo z'&\\ &1\enm\cdot \bemtwo \ell^{-\frac{1}{n-1}}\cdot I_{n-1}&\\ &\ell\enm\\
&=&\bemtwo \gamma&\gamma \overline{x}_n\\ &1\enm\cdot \bemtwo z'&\\ &1\enm\cdot \bemtwo \ell^{-\frac{1}{n-1}}\cdot I_{n-1}&\\ &\ell\enm\\
&=&\bemtwo \gamma&\gamma \overline{x}_n\\ &1\enm\cdot \bemtwo    \gamma^{-1}&\\ &1\enm  \cdot \bemtwo    \gamma&\\ &1\enm \cdot \bemtwo z'&\\ &1\enm\cdot \bemtwo \ell^{-\frac{1}{n-1}}\cdot I_{n-1}&\\ &\ell\enm\\
&=&\bemtwo I_{n-1}&\gamma \overline{x}_n\\ &1\enm\cdot \bemtwo \gamma z'&\\ &1\enm\cdot \bemtwo \ell^{-\frac{1}{n-1}}\cdot I_{n-1}&\\ &\ell\enm.
\end{eqnarray*}
This gives the decomposition 
\beq
P_{n-1,1}(\bbZ)\backslash X_n = SL_{n-1}(\bbZ)\backslash X_{n-1}\times (\bbR/\bbZ)^{n-1}\times\bbR^+.
\eeq

Let $F(z)$ be an automorphic function. We can unfold the following product:

\beq
&& \int_{SL_{n}\backslash X_n}  F(z) \ E_{n-1,1}\lp z,1,s\rp\  d^*z\\
&=&\int_{P_{n-1,1}\backslash X_n}  F(z) \det(Y)^s d^*z\\
&=&\int_{SL_{n-1}(\bbZ)\backslash X_{n-1}}\int_0^\infty \int_{(\bbZ/\bbR)^{n-1}} F(\overline{x}_n, z',\ell)  \ell^{-ns} \lp-\frac{n}{n-1} \rp \ d\overline{x}_n \   \ell^n\ \frac{d \ell}{\ell} \ d^*z' 
\eeq
Now define
\begin{eqnarray*}
G(z'):=\int_0^\infty \int_{(\bbZ/\bbR)^{n-1}} F(\overline{x}_n, z',\ell)  \ell^{-ns} \lp-\frac{n}{n-1} \rp \ d\overline{x}_n \   \ell^n\ \frac{d \ell}{\ell} 
\end{eqnarray*}

We claim that $G(z')$ is automorphic.

\begin{eqnarray*}
G(\gamma z') &=& \int_0^\infty \int_{(\bbZ/\bbR)^{n-1}} F(\overline{x}_n, \gamma z',\ell)  \ell^{-ns} \lp-\frac{n}{n-1} \rp \ d\overline{x}_n \   \ell^n\ \frac{d \ell}{\ell} 
\end{eqnarray*}
With the fact that $F$ is automorphic we have
\begin{eqnarray*}
&&F(\overline{x}_n, \gamma z',\ell) \\
&=&F\lp\bemtwo I_{n-1}& \overline{x}_n\\ &1\enm\cdot \bemtwo \gamma z'&\\ &1\enm\cdot \bemtwo \ell^{-\frac{1}{n-1}}\cdot I_{n-1}&\\ &\ell\enm\rp\\
&=&F\lp \bemtwo    \gamma&\\ &1\enm \bemtwo I_{n-1}& \gamma^{-1}\overline{x}_n\\ &1\enm\cdot \bemtwo  z'&\\ &1\enm\cdot \bemtwo \ell^{-\frac{1}{n-1}}\cdot I_{n-1}&\\ &\ell\enm\rp\\
&=&F\lp   \bemtwo I_{n-1}& \gamma^{-1}\overline{x}_n\\ &1\enm\cdot \bemtwo  z'&\\ &1\enm\cdot \bemtwo \ell^{-\frac{1}{n-1}}\cdot I_{n-1}&\\ &\ell\enm\rp.
\end{eqnarray*}
Now since $\gamma\in SL_{n-1}(\bbZ)$ and $F(z)$ is automorphic, the integral over $(\bbZ/\bbR)^{n-1}$ is invariant under the action of $\gamma$ on $\overline{x}_{n}$. So
\begin{eqnarray*}
G(\gamma z') &=& \int_0^\infty \int_{(\bbZ/\bbR)^{n-1}} F(\overline{x}_n, \gamma z',\ell)  \ell^{-ns} \lp-\frac{n}{n-1} \rp \ d\overline{x}_n \   \ell^n\ \frac{d \ell}{\ell} \\
&=& \int_0^\infty \int_{(\bbZ/\bbR)^{n-1}} F(\gamma^{-1}\overline{x}_n,  z',\ell)  \ell^{-ns} \lp-\frac{n}{n-1} \rp \ d\overline{x}_n \   \ell^n\ \frac{d \ell}{\ell} \\
&=&\int_0^\infty \int_{(\bbZ/\bbR)^{n-1}} F(\overline{x}_n,  z',\ell)  \ell^{-ns} \lp-\frac{n}{n-1} \rp \ d\overline{x}_n \   \ell^n\ \frac{d \ell}{\ell} \\
&=& G(z').
\end{eqnarray*}

\end{proof}

Now let $\nu=a+bi$. By the definition of incomplete Eisenstein series and Lemma \ref{changeof} we have
\beq
&&\int_{\quo}    E_{(n_1,\ldots,n_{r-1},1)}\lp z,\eta,\phi_1,\ldots,\phi_{r-1},1\rp \lbar   E_{(n-1,1)}\lp z,\nu,1\rp     \rbar^2 d^*z\\
&=&
 \int_{\quo}  \frac{1}{(2\pi i)^{r-1}}\int_{(2)}\cdots\int_{(2)} \tilde{\eta}(s_1,\ldots,s_{r-1})     E_{(n_1,\ldots,n_{r-1},1)}\lp z,s,\phi_1,\ldots,\phi_{r-1},1\rp      ds_1\cdots d{s_{r-1}}\\
 &&\times  \lbar   E_{(n-1,1)}\lp z,\nu,1\rp     \rbar^2 d^*z\\
 &=&-\frac{1}{\xi(n)} \int_{SL_{n-1}(\bbZ)\backslash X_{n-1}}\int_0^\infty \int_{(\bbZ/\bbR)^{n-1}}  \frac{1}{(2\pi i)^{r-1}}\int_{(2)}\cdots\int_{(2)} \tilde{\eta}(s_1,\ldots,s_{r-1})     \\
 && E_{(n_1,\ldots,n_{r-1})}\lp \frak{m}_{n-1}(z),s',\phi_1,\ldots,\phi_{r-1}\rp \ell ^{-n s''} \lbar   E_{(n-1,1)}\lp z,\nu,1\rp     \rbar^2 \ ds_1\cdots d{s_{r-1}} \ d^*z'  \  \prod_{h=1}^{n-1} dx_{h,n} \   \ell^n\ \frac{d \ell}{\ell}
\eeq
where $s' = (s'_1,\ldots,s'_{r-1})\in\bbC^{r-1}$ and $s''\in\bbC$ are given by
\beq
s'_i = s_i-\frac{n_1s_1+\cdots+n_{r-1}s_{r-1}}{n-1},&& s'' = \frac{n_1s_1+\cdots+n_{r-1}s_{r-1}}{n-1}.
\eeq
as before.

Now by Proposition \ref{constantsquare} we need to understand the following five integrals:

\begin{eqnarray*}
\Xi_1&=&\frac{-1}{\xi(n)} \mathop{\int}_{SL_{n-1}(\bbZ)\backslash X_{n-1}}\int_0^\infty    \frac{1}{(2\pi i)^{r-1}}\int_{(2)}\cdots\int_{(2)} \tilde{\eta}(s) E_{(n_1,\ldots,n_{r-1})}\lp z',s',\phi_1,\ldots,\phi_{r-1}\rp \ell ^{-n s''} \\
\nonumber&&\times  4 \ell^{-2na}  \ ds_1\cdots d{s_{r-1}} \ d^*z'  \     \ell^n\ \frac{d \ell}{\ell},
\end{eqnarray*}
\begin{eqnarray*}
\Xi_2&=&\frac{-1}{\xi(n)} \mathop{\int}_{SL_{n-1}(\bbZ)\backslash X_{n-1}}\int_0^\infty   \frac{1}{(2\pi i)^{r-1}}\int_{(2)}\cdots\int_{(2)} \tilde{\eta}(s) E_{(n_1,\ldots,n_{r-1})}\lp z',s',\phi_1,\ldots,\phi_{r-1}\rp  \ell ^{-n s''}\\
\nonumber&&\times \frac{\lbar \xi(n\nu-1) \rbar^2}{\lbar \xi(n\nu) \rbar^2}\ell^{-n\frac{2-a}{n-1}}  \lbar E_{(n-2,1)}\lp z',\frac{n\nu-1}{n-1} ,1   \rp\rbar^2  \ ds_1\cdots d{s_{r-1}} \ d^*z'  \     \ell^n\ \frac{d \ell}{\ell} ,
\end{eqnarray*}
\begin{eqnarray*}
\Xi_3&=&\frac{-1}{\xi(n)}\mathop{\int}_{SL_{n-1}(\bbZ)\backslash X_{n-1}}\int_0^\infty    \frac{1}{(2\pi i)^{r-1}}\int_{(2)}\cdots\int_{(2)} \tilde{\eta}(s) E_{(n_1,\ldots,n_{r-1})}\lp z',s',\phi_1,\ldots,\phi_{r-1}\rp  \ell ^{-n s''}\\
\nonumber&&\times \ell^{-\frac{n^2a-2na+n+in^2b}{n-1}}\frac{\xi(n\bar{\nu}-1)}{\xi(n\bar{\nu})} E_{(n-2,1)}\lp z',\frac{n\bar{\nu}-1}{n-1},1   \rp  \ ds_1\cdots d{s_{r-1}} \ d^*z'  \     \ell^n\ \frac{d \ell}{\ell},
\end{eqnarray*}
\begin{eqnarray*}
\Xi'_3&=&\frac{-1}{\xi(n)}\mathop{\int}_{SL_{n-1}(\bbZ)\backslash X_{n-1}}\int_0^\infty    \frac{1}{(2\pi i)^{r-1}}\int_{(2)}\cdots\int_{(2)} \tilde{\eta}(s) E_{(n_1,\ldots,n_{r-1})}\lp z',s',\phi_1,\ldots,\phi_{r-1}\rp  \ell ^{-n s''}\\
\nonumber&&\times \ell^{-\frac{n^2a-2na+n-in^2b}{n-1}}\frac{\xi(n{\nu}-1)}{\xi(n{\nu})} E_{(n-2,1)}\lp z',\frac{n{\nu}-1}{n-1},1   \rp  \ ds_1\cdots d{s_{r-1}} \ d^*z'  \     \ell^n\ \frac{d \ell}{\ell},
\end{eqnarray*}
\begin{eqnarray*}
\Xi_4&=&\frac{-1}{\xi(n)}\mathop{\int}_{SL_{n-1}(\bbZ)\backslash X_{n-1}}\int_0^\infty  \frac{1}{(2\pi i)^{r-1}}\int_{(2)}\cdots\int_{(2)} \tilde{\eta}(s) E_{(n_1,\ldots,n_{r-1})}\lp z',s',\phi_1,\ldots,\phi_{r-1}\rp\ell ^{-n s''}\\
\nonumber&&\times \frac{8\ell^{-2na}}{\lbar\xi \lp \nu\rp\rbar^2}\sum_{m_{n-1}=1}^\infty\sum_{\gamma\in P_{(n-2,1)}(\bbZ)\backslash SL_{n-1}(\bbZ)} m_{n-1}^{na-1}\sigma_{-na+1-inb}(m_{n-1})\sigma_{-na+1+inb}(m_{n-1}) \lp\frac{\ell^{-\frac{n}{n-1}}}{\det(z')^{\frac{1}{n-1}}}\rp^{-na+1}\\
\nonumber&&\times \ \left.K_{\frac{na-1}{2}+\frac{inb}{2}}\lp2\pi m_{n-1}\frac{\ell^{-\frac{n}{n-1}}}{\det(z')^{\frac{1}{n-1}}}\rp K_{\frac{na-1}{2} -\frac{inb}{2}}\lp2\pi m_{n-1}\frac{\ell^{-\frac{n}{n-1}}}{\det(z')^{\frac{1}{n-1}}}\rp \right|_\gamma   \ ds_1\cdots d{s_{r-1}} \ d^*z'  \     \ell^n\ \frac{d \ell}{\ell}.
\end{eqnarray*}

\begin{lem}
For $(n_1,\ldots,n_{r-1},1)\neq(1,\ldots,1)$, we have that $\Xi_1=0$.
\end{lem}

\begin{proof}
Since $(n_1,\ldots,n_{r-1},1)\neq(1,\ldots,1)$, the constant term of $E_{(n_1,\ldots,n_{r-1})}\lp z',s',\phi_1,\ldots,\phi_{r-1}\rp$ along the minimal parabolic is $0$. So we are integrating an automorphic function without constant term over the entire quotient $SL_{n-1}(\bbZ)\backslash X_{n-1}$. We must have $\Xi_1=0$.
\end{proof}

\begin{lem}
For $(n_1,\ldots,n_{r-1},1)\neq(1,\ldots,1)$, we have that $\Xi_3=\Xi'_3=0$.
\end{lem}
\begin{proof}
We can rearrange $\Xi_3$ and get
\beq
\Xi_3&=&\frac{-1}{\xi(n)} \frac{\xi(n\bar{\nu}-1)}{\xi(n\bar{\nu})}\mathop{\int}_{SL_{n-1}(\bbZ)\backslash X_{n-1}}\int_0^\infty    \frac{1}{(2\pi i)^{r-1}}\int_{(2)}\cdots\int_{(2)} \tilde{\eta}(s) E_{(n_1,\ldots,n_{r-1})}\lp z',s',\phi_1,\ldots,\phi_{r-1}\rp  \ell ^{-n s''}\\
\nonumber&&\times \ell^{-\frac{n^2a-2na+n+in^2b}{n-1}} E_{(n-2,1)}\lp z',\frac{n\bar{\nu}-1}{n-1},1   \rp  \ ds_1\cdots d{s_{r-1}} \ d^*z'  \     \ell^n\ \frac{d \ell}{\ell}
\eeq

We can now apply the Mellin inversion theorem to eliminate the integral over $\ell$. Thus $\Xi_3$ has the shape of
\beq
\Xi_3&=&c \mathop{\int}_{SL_{n-1}(\bbZ)\backslash X_{n-1}}\frac{1}{(2\pi i)^{r-2}}\int_{(2)}\cdots\int_{(2)}f(s') E_{(n_1,\ldots,n_{r-1})}\lp z',s',\phi_1,\ldots,\phi_{r-1}\rp \\
&&\times E_{(n-2,1)}\lp z',\nu',1   \rp  \ ds'_1\cdots d{s'_{r-2}} \ d^*z'
\eeq
for some $f$ with rapid decay in the imaginary parts of the arguments.
At this point, if $n_r\ge 2$, then unfolding with respect to $E_{(n-2,1)}\lp z',\nu',1   \rp$ and applying Proposition \ref{constant} reduce $\Xi_3$ to $0$. If $n_r=1$, we unfold with respect to $E_{(n_1,\ldots,n_{r-1})}\lp z',s',\phi_1,\ldots,\phi_{r-1}\rp$ and apply \ref{constantterm} repeatedly until the last partition number is no longer $1$ and apply Proposition \ref{constant}.
\end{proof}

\begin{lem}
For $(n_1,\ldots,n_{r-1},1)\neq(1,\ldots,1)$, we have that $\Xi_4=0$.
\end{lem}
\begin{proof}
\beq
\Xi_4&=&\frac{-1}{\xi(n)}\mathop{\int}_{SL_{n-1}(\bbZ)\backslash X_{n-1}}\int_0^\infty  \frac{1}{(2\pi i)^{r-1}}\int_{(2)}\cdots\int_{(2)} \tilde{\eta}(s) E_{(n_1,\ldots,n_{r-1})}\lp z',s',\phi_1,\ldots,\phi_{r-1}\rp\ell ^{-n s''}\\
\nonumber&&\times \frac{8\ell^{-2na}}{\lbar\xi \lp \nu\rp\rbar^2}\sum_{m_{n-1}=1}^\infty\sum_{\gamma\in P_{(n-2,1)}(\bbZ)\backslash SL_{n-1}(\bbZ)} m_{n-1}^{na-1}\sigma_{-na+1-inb}(m_{n-1})\sigma_{-na+1+inb}(m_{n-1}) \lp\frac{\ell^{-\frac{n}{n-1}}}{\det(z')^{\frac{1}{n-1}}}\rp^{-na+1}\\
\nonumber&&\times \ \left.K_{\frac{na-1}{2}+\frac{inb}{2}}\lp2\pi m_{n-1}\frac{\ell^{-\frac{n}{n-1}}}{\det(z')^{\frac{1}{n-1}}}\rp K_{\frac{na-1}{2} -\frac{inb}{2}}\lp2\pi m_{n-1}\frac{\ell^{-\frac{n}{n-1}}}{\det(z')^{\frac{1}{n-1}}}\rp \right|_\gamma   \ ds_1\cdots d{s_{r-1}} \ d^*z'  \     \ell^n\ \frac{d \ell}{\ell}.
\eeq
Recall that $\frac{\ell^{-\frac{n}{n-1}}}{\det(z')^{\frac{1}{n-1}}} = y_1$ and the effect of $\gamma$ on $y_1$ is given by part (3) of Lemma \ref{action}. We have $y'_{1} = y_{1}\lp b_1^2+\cdots + b_{n-1}^2\rp^{\frac{1}{2}}$ where $b_i's$ are defined in Lemma \ref{action}. We can then make the change of variable $m_{n-1}\frac{\ell^{-\frac{n}{n-1}}}{\det(z')^{\frac{1}{n-1}}} \lp b_1^2+\cdots + b_{n-1}^2\rp^{\frac{1}{2}} \mapsto u$. The sum of $\lp b_1^2+\cdots + b_{n-1}^2\rp^{\frac{1}{2}} $ over $P_{(n-2,1)}(\bbZ)\backslash SL_{n-1}(\bbZ)$ is simply a degenerate Eisenstein series. The sum over $m_{n-1}$ of divisor functions is a ratio of Zeta functions by Lemma \ref{divisor}. This together with ratio of Gamma functions (formed after integrating over $u$ of the $K$-Bessel functions) will give completed Zeta functions. We can now apply the previous Lemma which shows that $\Xi_4=0$.
\end{proof}

\begin{prop}\label{reduction}
For $(n_1,\ldots,n_{r-1},1)\neq(1,\ldots,1)$, we have that
\beq
&&\int_{\quo}    E_{(n_1,\ldots,n_{r-1},1)}\lp z,\eta,\phi_1,\ldots,\phi_{r-1},1\rp \lbar   E_{(n-1,1)}\lp z,\nu,1\rp     \rbar^2 d^*z\\
&=&C \frac{\lbar \xi(n\nu-1) \rbar^2}{ \lbar \xi(n\nu) \rbar^2}  \mathop{\int}_{SL_{n-1}(\bbZ)\backslash X_{n-1}} \frac{1}{(2\pi i)^{r-2}}\int_{(2)}\cdots\int_{(2)} \psi(s')E_{(n_1,\ldots,n_{r-1})}\lp z',s',\phi_1,\ldots,\phi_{r-1}\rp \\
&&\times  \lbar E_{(n-2,1)}\lp z',\frac{n\nu-1}{n-1} ,1   \rp\rbar^2  \ ds'_1\cdots d{s'_{r-2}} \ d^*z'
\eeq
for some absolute constant $C$.
\end{prop}

\begin{proof}From the previous three lemmas, we see that the only term left is $\Xi_2$. We have the following simplification for $\Xi_2$:
\beq
\Xi_2&=&\frac{-1}{\xi(n)} \mathop{\int}_{SL_{n-1}(\bbZ)\backslash X_{n-1}}\int_0^\infty   \frac{1}{(2\pi i)^{r-1}}\int_{(2)}\cdots\int_{(2)} \tilde{\eta}(s) E_{(n_1,\ldots,n_{r-1})}\lp z',s',\phi_1,\ldots,\phi_{r-1}\rp  \ell ^{-n s''}\\
\nonumber&&\times \frac{\lbar \xi(n\nu-1) \rbar^2}{\lbar \xi(n\nu) \rbar^2}\ell^{-n\frac{2-a}{n-1}}  \lbar E_{(n-2,1)}\lp z',\frac{n\nu-1}{n-1} ,1   \rp\rbar^2  \ ds_1\cdots d{s_{r-1}} \ d^*z'  \     \ell^n\ \frac{d \ell}{\ell}\\
&=&-\frac{n^2-n}{n_{r-1}}\frac{\lbar \xi(n\nu-1) \rbar^2}{\xi(n)\lbar \xi(n\nu) \rbar^2}  \mathop{\int}_{SL_{n-1}(\bbZ)\backslash X_{n-1}} \frac{1}{(2\pi i)^{r-2}}\int_{(2)}\cdots\int_{(2)} \psi(s')E_{(n_1,\ldots,n_{r-1})}\lp z',s',\phi_1,\ldots,\phi_{r-1}\rp \\
&&\times  \lbar E_{(n-2,1)}\lp z',\frac{n\nu-1}{n-1} ,1   \rp\rbar^2  \ ds'_1\cdots d{s'_{r-2}} \ d^*z',
\eeq
where the factor $\frac{n^2-n}{n_{r-1}}$ comes from the Jacobian of the change of variables and
\beq
\psi(s'):= \tilde{\eta}\lp s_1'+c,\ldots,s'_{r-2}+c,-(s'_1+\ldots+s'_{r-2}+c)\rp
\eeq
where $c=\frac{2-a}{n-1}-1$.
\end{proof}

This proposition can be used until we reach the following:
\begin{enumerate}
\item 
\beq
  \int_{\sltz\backslash\bbH}\phi(z)\lbar E\lp z,s\rp \rbar^2 \ d^*z
\eeq
where $\phi$ is a cusp form on $GL(2)$ and $E\lp z,s\rp$ is the Eisenstein series on $GL(2)$.
\item 
\beq
  \int_{\quo}\phi(z)\lbar E_{n-1,1}\lp z,s,1\rp \rbar^2 \ d^*z
\eeq
where $\phi$ is a cusp form on $GL(n)$ for $n\ge3$.
\item 
\beq
\int_{\quo}    E_{(n_1,\ldots,n_{r})}\lp z,\eta,\phi_1,\ldots,\phi_{r}\rp \lbar   E_{(n-1,1)}\lp z,\nu,1\rp     \rbar^2 d^*z
\eeq
for $n=n_1+\ldots+n_r$ in non-increasing order with $n_r\ge2$.

We will consider these three separately in the next three sections.
\end{enumerate}

 \vspace{20mm}
\subsection{A \texorpdfstring{$GL(2)$}{Lg} Calculation}
In this section, our goal is to show the following estimate.
\begin{thm}\label{prop6}
\beq
\int_{\quo}E_{(2,1,\ldots,1)}(z,\eta,\phi)\lbar E_{n-1,1}\lp z,\frac{1}{2}+it,1\rp\rbar^2d^*z \ll t^{-\frac{1}{2}+\epsilon}.
\eeq
\end{thm}
\begin{rem}
We get this rate by only utilizing the convexity bounds for the Zeta function and L-functions associated to Maass forms.
\end{rem}

\begin{proof}

First we apply Proposition \ref{reduction} $(n-2)$ times and get
\beq
&&\int_{\quo}E_{(2,1,\ldots,1)}(z,\eta,\phi)\lbar E_{n-1,1}\lp z,\frac{1}{2}+it,1\rp\rbar^2d^*z\\
&\ll& \frac{\lbar \xi\lp 2-\frac{n}{2}+int\rp\rbar^2}{\lbar \xi\lp \frac{n}{2}+int\rp\rbar^2}   \int_{\sltz\backslash X_2}  \phi(z)\lbar E\lp z,1-\frac{n}{4}+int\rp\rbar^2d^*z
\eeq

We follow the proof of Proposition 2.1 of \cite{LS} closely.
Consider
\beq
I_{\phi}(s)&:=& \int_{\sltz\backslash X_2}\phi(z)   E\lp z,1-\frac{n}{4}+int\rp E(z,s)  \ d^*z\\
&=&\int_0^\infty  \int_0^1 \phi(z)  E\lp z,1-\frac{n}{4}+int\rp y^s\frac{dxdy}{y^2},
\eeq
we will subsitute $s=1-\frac{n}{4}-int$ later in the proof. For odd forms this integral is $0$, so we only consider even forms.  Fourier expansions of even forms $\phi(z)$ and $E(z,s)$ are given by:
\beq
\phi(z)=y^{\frac{1}{2}} \sum_{n=1}^\infty\rho_\phi(1)\lambda_\phi(n)K_{it_\phi}(2\pi ny)\cos(2\pi nx),
\eeq
\beq
E(z,s) = y^s+\frac{\xi(2s-1)}{\xi(2s)}y^{1-s}+\frac{2y^{\frac{1}{2}}}{\xi(2s)}\sum_{n=1}^\infty n^{s-\frac{1}{2}}\sigma_{1-2s}(n)K_{s-\frac{1}{2}}(2\pi ny)\cos(2\pi nx).
\eeq
Since $\phi$ is fixed, we can ignore the normalization constant $\rho_\phi(1)$. The spectral parameter $\lambda_\phi$ is related to $t_\phi$ by $\lambda_\phi = \frac{1}{4}+t_\phi^2$. The $L$-function associated to $u_\phi$ is defined by
\beq
L(\phi,s):=\sum_{n=1}^\infty\frac{\lambda_\phi(n)}{n^s} = \prod_p(1-\lambda_\phi(p)p^{-s}+p^{-2s})^{-1}.
\eeq
So
\beq
I_\phi(s)&=&\int_0^\infty  \int_0^1  \lp y^{\frac{1}{2}} \sum_{n=1}^\infty \lambda_\phi(n)K_{it_\phi}(2\pi ny)\cos(2\pi nx)\rp  \lp  y^{1-\frac{n}{4}+int}+\frac{\xi(1-\frac{n}{2}+2int)}{\xi(2-\frac{n}{2}+2int)}y^{\frac{n}{4}-int} \right.\\  &&+\left.\frac{2y^{\frac{1}{2}}}{\xi(2-\frac{n}{2}+2int)}\sum_{n'=1}^\infty n'^{\frac{1}{2}-\frac{n}{4}+int}\sigma_{\frac{n}{2}-1-2int}(n')K_{\frac{1}{2}-\frac{n}{4}+int}(2\pi n'y)\cos(2\pi n'x)\rp y^s\frac{dxdy}{y^2}\\
&=&\frac{1}{\xi(2-\frac{n}{2}+2int)}  \int_0^\infty \sum_{n=1}^\infty \lambda_\phi(n)K_{it_\phi}(2\pi ny)n^{\frac{1}{2}-\frac{n}{4}+int}\sigma_{\frac{n}{2}-1-2int}(n)K_{\frac{1}{2}-\frac{n}{4}+int}(2\pi ny)y^s \frac{dy}{y}\\
&=&\frac{1}{\xi(2-\frac{n}{2}+2int)}  \lp \sum_{n=1}^\infty \frac{\lambda_\phi(n) n^{\frac{1}{2}-\frac{n}{4}+int}\sigma_{\frac{n}{2}-1-2int}(n)}{n^s}   \rp  \int_0^\infty  K_{it_\phi}(2\pi y)  K_{\frac{1}{2}-\frac{n}{4}+int}(2\pi y) y^s \frac{dy}{y}.
\eeq
Following equation (15) of \cite{LS}, it can be shown that
\beq
\sum_{n=1}^\infty \frac{\lambda_\phi(n) n^{\frac{1}{2}-\frac{n}{4}+int}\sigma_{\frac{n}{2}-1-2int}(n)}{n^s}  =\frac{L(\phi,s-\frac{1}{2}+\frac{n}{4}-int)L(\phi, s+\frac{1}{2}-\frac{n}{4}+int)}{\zeta(2s)}.
\eeq
The Mellin transform of product of $K$-Bessel functions can be evaluated using \ref{stade}, we have
\beq
&& \int_0^\infty  K_{it_\phi}(2\pi y)  K_{\frac{1}{2}-\frac{n}{4}+int}(2\pi y) y^s \frac{dy}{y}\\
&=&\frac{2^{-3}\pi^{-s}}{\Gamma(s)}\Gamma\lp\frac{s+it_\phi+\frac{1}{2}-\frac{n}{4}+int}{2}\rp\Gamma\lp\frac{s-it_\phi+\frac{1}{2}-\frac{n}{4}+int}{2}\rp\Gamma\lp\frac{s+it_\phi-\frac{1}{2}+\frac{n}{4}-int}{2}\rp\\&&\times\Gamma\lp\frac{s-it_\phi-\frac{1}{2}+\frac{n}{4}-int}{2}\rp.
\eeq
Now let $s=1-\frac{n}{4}-int$, together we can estimate the following product by using Stirling's formula and the bound $ L(u_j,\frac{1}{2}+it)\ll |t|^{\frac{1}{2}+\epsilon}$:
\beq
&& \frac{\lbar \xi\lp 2-\frac{n}{2}+int\rp\rbar^2}{\lbar \xi\lp \frac{n}{2}+int\rp\rbar^2} \frac{L(\phi,\frac{1}{2}-2int)L(\phi,\frac{3-n}{2})}{\xi(2-\frac{n}{2}+2int) \xi(2-\frac{n}{2}-2int)}  \Gamma\lp\frac{\frac{3-n}{2}+it_\phi}{2}\rp\Gamma\lp\frac{\frac{3-n}{2}-it_\phi}{2}\rp  \\&&  \times\Gamma\lp\frac{\frac{1}{2}-2int+it_\phi}{2}\rp  \Gamma\lp\frac{\frac{1}{2}-2int-it_\phi}{2}\rp\\
&\ll_{\phi,\epsilon}&t^{-1}t^{\frac{1}{2}+\epsilon} =t^{-\frac{1}{2}+\epsilon}.
\eeq

\end{proof}

\vspace{20mm}

\subsection{Cuspidal Contribution}

\begin{prop}
\beq
\int_{ \quo} \phi(z) E_{n-1,1}(z,s,1) E_{n-1,1}(z,s',1) d^*z =0.
\eeq
for $n\ge3$ and $\phi$ a Maass cusp form on $GL(n)$.
\end{prop}

\begin{proof}
Let $\phi(z)$ be a cusp form with Fourier-Whittaker expansion
\beq
\phi(z) = \sum_{\gamma\in U_{n-1}(\bbZ)\backslash SL_{n-1}(\bbZ)}\sum_{m_1=1}^\infty\cdots\sum_{m_{n-2}=1}^\infty \sum_{m_{n-1}\neq 0}\left.\frac{A(m_1,\ldots,m_{n-1})}{\prod_{k=1}^{n-1}|m_k|^{k(n-k)/2}} \wjac(My,\nu) e(m_1x_1+\cdots+m_{n-1}x_{n-1})\right|_\gamma
\eeq
where 
\beq
M = \bemfour m_1\cdots |m_n|&&&\\ &\ddots&&\\ &&m_1&\\&&&1\enm.
\eeq
Note we define the slash operator to be such that 
\beq
f(z)|_\gamma = f\lp\bemtwo \gamma & \\ &1\enm \cdot z\rp.
\eeq
By the Rankin-Selberg convolution on $GL(n)$,
\beq
&&\int_{SL_n(\bbZ)\backslash X_n}\phi(z) E_n(z,s',1)E_n(z,s,1)  d^*z\\
&=&\int_{-\infty}^\infty \cdots \int_{-\infty}^\infty \int_0^1\cdots\int_0^1 \sum_{m_1=1}^\infty\cdots\sum_{m_{n-2}=1}^\infty \sum_{m_{n-1}\neq 0}\frac{A(m_1,\ldots,m_{n-1})}{\prod_{k=1}^{n-1}|m_k|^{k(n-k)/2}} \wjac(My,\nu) e(m_1x_1+\cdots+m_{n-1}x_{n-1}) \\
&&\ \times \ E_n(z,s',1)\  \det(y)^s \ \prod_{1\le i<j \le n-1} dx_{i,j}\  \prod_{k=1}^{n-1}\frac{dy_k}{y_k^{k(n-k)+1}}.
\eeq
To utilize the Fourier expansion
\beq
E_n(z,s,1) = \sum_{m_1\in\bbZ}\hat{\phi}_{(m_1,0,\ldots,0)}(z)+\sum_{i=2}^{n-1}\sum_{\gamma_i\in P_{i}(\bbZ)\backslash SL_i(\bbZ)} \sum_{m_i=1}^\infty \hat{\phi}_{(0,\ldots,0,m_i,0\ldots,0)}\lp\bemtwo \gamma_i &\\ &I_{n-i}\enm z\rp,
\eeq 
we must understand the effect of $\bemtwo \gamma_i&\\&I_{n-i}\enm$ on $x_{i,i+1}$. First we write
\beq
x = \bemtwo X_1&V\\ &X_2\enm && y= \bemtwo Y_1 y_1\cdots y_{n-i-1}&\\& Y_2\enm,
\eeq
where $x\cdot y$ is the Iwasawa decomposition of $z$ and $X_1,Y_1$ are of dimension $i\times i$, $X_2,Y_2$ are of dimension $(n-i)\times(n-i)$.
Matrix multiplication gives
\beq
\bemtwo \gamma_i&\\&I_{n-i}\enm  \bemtwo X_1&V\\ &X_2\enm \bemtwo Y_1 y_1\cdots y_{n-i-1}&\\& Y_2\enm =  \bemtwo\gamma_i \cdot X_1 \cdot Y_1&\gamma_i \cdot V \\&X_2\enm \bemtwo  I_i y_1\cdots y_{n-i-1}&\\ Y_2\enm.
\eeq
By the Iwasawa decomposition, the matrix $\gamma_i\cdot X_1\cdot Y_1$ can be written as $X_1'\cdot Y_1'\cdot K_1'$. So
\beq
\bemtwo \gamma_i&\\&I_{n-i}\enm z =  \bemtwo X_1'& \gamma_i\cdot V\\& X_2 \enm    \bemtwo Y_1' y_1\cdots y_{n-i-1}&\\& Y_2\enm  \bemtwo K'&  \\&1\enm.
\eeq
We have put $\bemtwo \gamma_i&\\&I_{n-i}\enm z $ in Iwasawa form. In particular the $(i,i+1)$-entry becomes
\beq
\sum_{\ell=1}^i \gamma_{i,\ell} x_{\ell,i+1},
\eeq
which appears in the bottom left corner of $\gamma_i\cdot V$.

It is important to notice that $X_1'$ and $Y_1'$ are completed determined by $X_1,Y_1$ and $\gamma_i$. They do not have any relations to the matrix $V$. Thus the exponential factor $e(m_ix_{i,i+1})$ in $\hat{\phi}_{(0,\ldots,0,m_i,0\ldots,0)}(z)$ is the only place  $x_{i,i+1}$ appears. As a result $\sum_{\ell=1}^i \gamma_{i,\ell} x_{\ell,i+1} $ only appears in the exponential after applying $\gamma_i$. Integrating against $x_{\ell,i+1}$ forces $\gamma_{\ell,i+1}$ to be $0$ for $1\le \ell\le i-1$. This also means that $\gamma_{i,i}=1$ as $\gamma\in SL_{n-1}(\bbZ)$. So only the identity coset in the sum over $P_i(\bbZ)\backslash SL_i(\bbZ)$ contributes. Finally integrating against either $x_{i+1,i+2}$ or $x_{i-1,i}$ shows the cuspidal contribution is $0$. This completes the proof.
\end{proof}

\vspace{20mm}

\subsection{Partition of type \texorpdfstring{$n=n_1+\cdots+n_r$, $n_r\ge2$}{Lg}.}
\ \ The goal of this section to prove the following proposition.
\begin{prop}
Let $n=n_1+\cdots+n_r$ be a partition of $n$ in non-increasing order with $n_r\ge2$, we have 
\beq
\int_{SL_n(\bbZ)\backslash X_n}   E_{{n_1,\ldots,n_r}}\lp z,,\eta,\phi_1,\ldots,\phi_r\rp \lbar E_{n-1,1}\lp z,s,1\rp \rbar^2d^*z = 0.
\eeq
\end{prop}
\begin{proof}
By Lemma \ref{changeof}, the triple product can be written as
\begin{eqnarray*}
&&\int_{SL_n(\bbZ)\backslash X_n}  E_{{n_1,\ldots,n_r}}\lp z,\phi_1,\ldots,\phi_r,\eta\rp   \overline{E_{n-1,1}}\lp z,1,s\rp E_{n-1,1}\lp z,1,s\rp d^*z\\
&=& \int_{SL_{n-1}(\bbZ)\backslash X_{n-1}} \int_0^\infty \int_{(\bbZ/\bbR)^{n-1}} E_{{n_1,\ldots,n_r}}\lp z,\phi_1,\ldots,\phi_r,\eta\rp   \overline{E_{n-1,1}}\lp z,1,s\rp   \ell^{-ns}\lp-\frac{n}{n-1} \rp \ d\overline{x}_n \   \ell^n\ \frac{d \ell}{\ell}\  d^*z',
\end{eqnarray*}
where
\begin{eqnarray*}
\Xi(z') := \int_0^\infty \int_{(\bbZ/\bbR)^{n-1}} E_{{n_1,\ldots,n_r}}\lp \overline{x}_n,  z',\ell,\phi_1,\ldots,\phi_r,\eta\rp   \overline{E_{n-1,1}}\lp \overline{x}_n,  z',\ell,1,s\rp   \ell^{-ns}\lp-\frac{n}{n-1} \rp \ d\overline{x}_n \   \ell^n\ \frac{d \ell}{\ell}
\end{eqnarray*}
is automorphic.

We have the Fourier expansion:
\beq
&&E_{{n_1,\ldots,n_r}}\lp z,\phi_1,\ldots,\phi_r,\eta \rp\\
&=&\sum_{m_1=0}^\infty \sum_{m_2=0}^\infty {\psum_{\gamma_2\in P_{1,1}\backslash SL_2}}\cdots \sum_{m_{n-1}=1}^\infty \sum_{\gamma_{n-1}\in P_{n-2,1}\backslash SL_{n-1}} a_{m_1,\ldots,m_{n-1}}\\ &&\times \left.W_{m_1,\ldots,m_{n-1}}(Y) e\lp m_1x_{1,2}+\cdots m_{n-1} x_{n-1,n}\rp \rbar_{\gamma_2\cdots \gamma_{n-1}}
\eeq

By Lemma \ref{coezero} we have that

\begin{enumerate}
\item
$a_{m_1,\ldots,m_{n-2},0} = 0$, hence, in the Fourier expansion above, the variable $m_{n-1}$ starts at $1$.
\item
$a_{m_1,\ldots,m_{n-4},0,0,m_{n-1}}=0$.
\item
$a_{0,\ldots,0}=0$.
\end{enumerate}

We now use these properties to simplify $\Xi(z')$. Notice  $m_{n-1} \neq 0$ in the expansion for\\ $E_{{n_1,\ldots,n_r}}\lp z,\phi_1,\ldots,\phi_r,\eta \rp$, so $E_{{n_1,\ldots,n_r}}\lp z,\phi_1,\ldots,\phi_r,\eta \rp$ is orthogonal to all terms with $i<n-1$ in the Fourier expansion of $E_{(n-1,1)}(z,s,1)$. Thus after executing the $\overline{x}_{n}$ integrals we have
\begin{eqnarray}
\nonumber \Xi(z')&=&\int_0^\infty \sum_{m_1=0}^\infty \sum_{m_2=0}^\infty {\psum_{\gamma_2\in P_{1,1}\backslash SL_2}}\cdots \sum_{m_{n-1}=1}^\infty \sum_{\gamma_{n-1}\in P_{n-2,1}\backslash SL_{n-1}} a_{m_1,\ldots,m_{n-1}} b_{m_{n-1}}\\   &&\times \left.W_{m_1,\ldots,m_{n-1}}(Y',\ell ) e\lp m_1x_{1,2}+\cdots m_{n-2} x_{n-2,n-1}\rp \rbar_{\gamma_2\cdots \gamma_{n-1}}   \left. W'_{m_{n-1}}(Y',\ell)\right|_{\gamma_{n-1}}  \lp-\frac{n}{n-1} \rp  \   \ell^n\ \frac{d \ell}{\ell}. \label{imp}
\end{eqnarray}
Here $z'=X'Y'$ in Iwasawa form and we used the simplified notation for the Fourier coefficients of degenerate Eisenstein series:
\beq
b_{m_{n-1}}&=&\frac{2}{\xi(ns)}|m_{n-1}|^{\frac{ns}{2}-\frac{1}{2}}\sigma_{-ns+1}(|m_{n-1}|)\\
W'(Y',\ell)&=&K_{\frac{ns}{2}-\frac{1}{2}}(2\pi |m_{n-1}|y_{1})   \lp y_{1}^{n-1}y_{2}^{n-2}\cdots y_{n-1}\rp^s   y_{1}^{-\frac{ns}{2}+\frac{1}{2}}\\
&=&K_{\frac{ns}{2}-\frac{1}{2}}(2\pi |m_{n-1}|y_{1})   \ell^{-ns}   y_{1}^{-\frac{ns}{2}+\frac{1}{2}}. 
\eeq
This corresponds to the term $\hat{\phi}_{(0,\ldots,0,m_{n-1})}$ in Theorem \ref{prop3}. It is important to note that $\ell$ is invariant under the action of $\gamma_2,\ldots,\gamma_{n-1}$, so we may move the integral inside all the summations:
\beq
\Xi (z')&=&  \sum_{m_1=0}^\infty \sum_{m_2=0}^\infty {\psum_{\gamma_2\in P_{1,1}\backslash SL_2}}\cdots \sum_{m_{n-1}=1}^\infty \sum_{\gamma_{n-1}\in P_{n-2,1}\backslash SL_{n-1}} a_{m_1,\ldots,m_{n-1}} b_{m_{n-1}}\\   &&\times \int_0^\infty\left.W_{m_1,\ldots,m_{n-1}}(Y',\ell ) e\lp m_1x_{1,2}+\cdots m_{n-2} x_{n-2,n-1}\rp \rbar_{\gamma_2\cdots \gamma_{n-1}}   \left. W'_{m_{n-1}}(Y',\ell)\right|_{\gamma_{n-1}}  \lp-\frac{n}{n-1} \rp  \   \ell^n\ \frac{d \ell}{\ell}.
\eeq

It's tempting to conclude here that $\Xi(z')$ is an automorphic function without constant term, so

 $\int_{SL_{n-1}\backslash X_{n-1}} \Xi(z')d^*z'$ must be $0$. However this is not the correct form of Fourier expansion for $\Xi(z')$. Note that $\Xi(z')$ is a function on $SL_{n-1}\backslash X_{n-1}$. The correct form of Fourier expansion looks like
\beq
\sum_{m_1=0}^\infty \sum_{m_2=0}^\infty {\psum_{\gamma_2\in P_{1,1}\backslash SL_2}}\cdots \sum_{m_{n-2}=0}^\infty \psum_{\gamma_{n-2}\in P_{n-3,1}\backslash SL_{n-2}}    a_{(m_1,\ldots,m_{n-2})} f_{(m_1,\ldots,m_{n-2})}(y) e(m_{1}x_{1,2}+\ldots+m_{n-2}x_{n-2,n-1}).
\eeq

In our expansion of $\Xi(z')$, we have an sum over $P_{n-2,1}\backslash SL_{n-1}$. But this allows us to unfold again.
\beq
&&\int_{SL_{n-1}\backslash X_{n-1}} \Xi(z')d^*z'\\
&=&\int_{P_{n-2,1}\backslash X_{n-1} }   \sum_{m_1=0}^\infty \sum_{m_2=0}^\infty {\psum_{\gamma_2\in P_{1,1}\backslash SL_2}}\cdots \sum_{m_{n-2}=0}^\infty \psum_{\gamma_{n-2}\in P_{n-3,1}\backslash SL_{n-2}} \sum_{m_{n-1}=1}^\infty a_{m_1,\ldots,m_{n-1}} b_{m_{n-1}}\\   &&\times \int_0^\infty\left.W_{m_1,\ldots,m_{n-1}}(Y',\ell ) e\lp m_1x_{1,2}+\cdots m_{n-2} x_{n-2,n-1}\rp \rbar_{\gamma_2\cdots \gamma_{n-2}}   W'_{m_{n-1}}(Y',\ell)   \lp-\frac{n}{n-1} \rp  \   \ell^n\ \frac{d \ell}{\ell} \  d^*z'
\eeq

Now suppose $F(z)$ is an automorphic function on $ SL_{n}(\bbZ) \backslash X_n$ in the form of $F(z) = \sum_{\gamma\in P_{n-1,1}\backslash SL_{n}(\bbZ)} \left.G(z)\right|_\gamma$. First of all, for this definition to make sense, we need $G(z)$ to be invariant under the group $P_{n-1,1}$. (This is clearly the case for $\Xi(z')$.) Then
\beq
&&\int_{SL_{n}(\bbZ) \backslash X_n} F(z) d^*z\\
&=&\int_{ P_{n-1,1} \backslash X_n} G(z) d^*z\\
&=&\int_{SL_{n-1}(\bbZ) \backslash X_{n-1}}  \int_0^\infty \int_{(\bbZ/\bbR)^{n-1}}   G\lp \bemtwo I_{n-1}& \overline{x}_n\\ &1\enm\cdot \bemtwo  z'&\\ &1\enm\cdot \bemtwo \ell^{-\frac{1}{n-1}}\cdot I_{n-1}&\\ &\ell\enm \rp \\ &&\times \lp-\frac{n}{n-1} \rp \ d\overline{x}_n   \ell^n\frac{d \ell}{\ell} d^*z'.\\
\eeq
We must verify that
\beq
G'(z'):= \int_0^\infty \int_{(\bbZ/\bbR)^{n-1}}   G\lp \bemtwo I_{n-1}& \overline{x}_n\\ &1\enm\cdot \bemtwo  z'&\\ &1\enm\cdot \bemtwo \ell^{-\frac{1}{n-1}}\cdot I_{n-1}&\\ &\ell\enm \rp \lp-\frac{n}{n-1} \rp \ d\overline{x}_n   \ell^n\frac{d \ell}{\ell}
\eeq
is automorphic. Indeed,

\beq
&&G'(\gamma z') \\
&=& \int_0^\infty \int_{(\bbZ/\bbR)^{n-1}}   W\lp \bemtwo I_{n-1}& \overline{x}_n\\ &1\enm\cdot \bemtwo  \gamma&\\ &1\enm\cdot \bemtwo  z'&\\ &1\enm\cdot \bemtwo \ell^{-\frac{1}{n-1}}\cdot I_{n-1}&\\ &\ell\enm \rp \lp-\frac{n}{n-1} \rp \ d\overline{x}_n   \ell^n\frac{d \ell}{\ell}\\
&=& \int_0^\infty \int_{(\bbZ/\bbR)^{n-1}}   W\lp \bemtwo I_{n-1}& \gamma^{-1}\overline{x}_n\\ &1\enm \cdot \bemtwo  z'&\\ &1\enm\cdot \bemtwo \ell^{-\frac{1}{n-1}}\cdot I_{n-1}&\\ &\ell\enm \rp \lp-\frac{n}{n-1} \rp \ d\overline{x}_n   \ell^n\frac{d \ell}{\ell}\\
&=& \int_0^\infty \int_{(\bbZ/\bbR)^{n-1}}   W\lp \bemtwo I_{n-1}& \overline{x}_n\\ &1\enm \cdot \bemtwo  z'&\\ &1\enm\cdot \bemtwo \ell^{-\frac{1}{n-1}}\cdot I_{n-1}&\\ &\ell\enm \rp \lp-\frac{n}{n-1} \rp \ d\overline{x}_n   \ell^n\frac{d \ell}{\ell}\\
&=&G'(z').
\eeq

So 
\beq
&&\int_{SL_{n-1}\backslash X_{n-1}} \Xi(z')d^*z'\\
&=&\int_{P_{n-2,1}\backslash X_{n-1} }   \sum_{m_1=0}^\infty \sum_{m_2=0}^\infty {\psum_{\gamma_2\in P_{1,1}\backslash SL_2}}\cdots \sum_{m_{n-2}=0}^\infty \psum_{\gamma_{n-2}\in P_{n-3,1}\backslash SL_{n-2}} \sum_{m_{n-1}=1}^\infty a_{m_1,\ldots,m_{n-1}} b_{m_{n-1}}\\   &&\times \int_0^\infty\left.W_{m_1,\ldots,m_{n-1}}(Y',\ell ) e\lp m_1x_{1,2}+\cdots m_{n-2} x_{n-2,n-1}\rp \rbar_{\gamma_2\cdots \gamma_{n-2}}   W'_{m_{n-1}}(Y',\ell)   \lp-\frac{n}{n-1} \rp  \   \ell^n\ \frac{d \ell}{\ell} \  d^*z'\\
&=&\int_{SL_{n-2}\backslash X_{n-2} }   \sum_{m_1=0}^\infty \sum_{m_2=0}^\infty {\psum_{\gamma_2\in P_{1,1}\backslash SL_2}}\cdots \sum_{m_{n-3}=1}^\infty \sum_{\gamma_{n-3}\in P_{n-4,1}\backslash SL_{n-3}} \sum_{m_{n-1}=1}^\infty a_{m_1,\ldots,m_{n-3},0,m_{n-1}} b_{m_{n-1}}\\   &&\times \int_0^\infty \int_0^\infty\left.W_{m_1,\ldots,m_{n-3},0,m_{n-1}}(Y'',\ell',\ell ) e\lp m_1x_{1,2}+\cdots +m_{n-3} x_{n-3,n-2}\rp \rbar_{\gamma_2\cdots \gamma_{n-3}}  \\&& \times W'_{m_{n-1}}(\ell',\ell)   \lp-\frac{n}{n-1} \rp  \   \ell^n\ \frac{d \ell}{\ell}     \lp-\frac{n-1}{n-2} \rp  \   \ell'^n\ \frac{d \ell'}{\ell'}  d^*z'\\
\eeq
with 
\beq
z'= \bemtwo I_{n-2}& \overline{x}_{n-1}\\ &1\enm\cdot \bemtwo z''&\\ &1\enm\cdot \bemtwo \ell'^{-\frac{1}{n-2}}\cdot I_{n-2}&\\ &\ell'\enm
\eeq
and $\ell' = \lp y_2^{n-2}\cdots y_{n-1}\rp^{-\frac{1}{n-1}}$.

Here we also note that $W'(Y',\ell)$ is really a function of $y_1$ and $\ell$. But $y_1= \ell^{-\frac{n}{n-1}}\ell'$ so $W'(Y',\ell)$ is a function of $\ell$ and $\ell'$. And $\ell$ and $\ell'$ are invariant under the action of $\gamma_2,\ldots, \gamma_{n-3}$.

Finally, 
\beq
&& \sum_{m_1=0}^\infty \sum_{m_2=0}^\infty {\psum_{\gamma_2\in P_{1,1}\backslash SL_2}}\cdots \sum_{m_{n-3}=1}^\infty \sum_{\gamma_{n-3}\in P_{n-4,1}\backslash SL_{n-3}} \sum_{m_{n-1}=1}^\infty a_{m_1,\ldots,m_{n-3},0,m_{n-1}} b_{m_{n-1}}\\   &&\times \int_0^\infty \int_0^\infty W_{m_1,\ldots,m_{n-3},0,m_{n-1}}(Y'',\ell',\ell )   W'_{m_{n-1}}(\ell',\ell)   \lp-\frac{n}{n-1} \rp  \   \ell^n\ \frac{d \ell}{\ell}   \lp-\frac{n-1}{n-2} \rp  \   \ell'^n\ \frac{d \ell'}{\ell'}  \\&&\left.\times     e\lp m_1x_{1,2}+\cdots +m_{n-3} x_{n-3,n-2}\rp \rbar_{\gamma_2\cdots \gamma_{n-3}} 
\eeq
is in proper Fourier expansion form and is automorphic without constant term. Thus the integral over $SL_{n-2}\backslash X_{n-2}$ is $0$.

\end{proof}

\newpage

\section{Main Term From Minimal Parabolic Contribution}

The purpose of this section is to prove the following theorem.
\begin{thm}\label{mainterm}
\beq &&\int_{SL_n(\bbZ)\backslash X_n}  E_{(1,\ldots,1)}(z,\eta) \lbar   E_{(n-1,1)}\lp z,\frac{1}{2}+it,1\rp     \rbar^2 d^*z\\
&=& \frac{2}{\xi(n)} \log(t)\lp\int_{SL_n(\bbZ)\backslash X_n} E_{(1,\ldots,1)}(z,\eta) d^*z\rp + O_{n,\eta}(1),
\eeq
as $t\rightarrow \infty$.
\end{thm}

We begin the proof by evaluating the integral on the right hand side.

\begin{prop}
\beq
\int_{SL_n(\bbZ)\backslash X_n} E_{(1,\ldots,1)}(z,\eta) d^*z = \frac{1}{n^{n-2}}\tilde{\eta}\lp\frac{2}{n},\ldots,\frac{2}{n}\rp.
\eeq
\end{prop}
\begin{proof}
Recall that the incomplete Eisenstein series associated to the minimal parabolic subgroup is given by
\beq
E_{(1,\ldots,1)}(z,\eta)  = \sum_{\gamma\in P_{1,\ldots,1}(\bbZ)\backslash SL_{n}(\bbZ)}  \left.\eta\lp \prod_{j=1}^{n-1} y_j^{b_{j,1}},\prod_{j=1}^{n-1} y_j^{b_{j,2}},\ldots,\prod_{j=1}^{n-1} y_j^{b_{j,n-1}}\rp\rbar_{\gamma}
\eeq
where
\beq
b_{j,k} = \lbra \begin{array}{ll}jk & \mbox{if }j+k\le n,\\(n-j)(n-k) &  \mbox{if }j+k\ge n.\end{array}\right.
\eeq
We can unfold with respect to the sum over $P_{1,\ldots,1}(\bbZ)\backslash SL_{n}(\bbZ)$:
\beq
&&\int_{SL_n(\bbZ)\backslash X_n} E_{(1,\ldots,1)}(z,\eta) d^*z\\
&=&\int_{P_{1,\ldots,1}(\bbZ)\backslash X_n}\eta\lp \prod_{j=1}^{n-1} y_j^{b_{j,1}},\prod_{j=1}^{n-1} y_j^{b_{j,2}},\ldots,\prod_{j=1}^{n-1} y_j^{b_{j,n-1}}\rp d^*z\\
&=& \int_{{(\bbR^{+})}^{n-1}}\int_{N_{min}(\bbZ)\backslash N_{min}(\bbR)}\eta\lp \prod_{j=1}^{n-1} y_j^{b_{j,1}},\prod_{j=1}^{n-1} y_j^{b_{j,2}},\ldots,\prod_{j=1}^{n-1} y_j^{b_{j,n-1}}\rp \prod_{1\le \alpha<\beta\le n} dx_{\alpha,\beta} \prod_{\ell = 1}^{n-1} y_{\ell}^{-\ell(n-\ell)-1}dy_{\ell}\\
&=& \int_{{(\bbR^{+})}^{n-1}}\eta\lp \prod_{j=1}^{n-1} y_j^{b_{j,1}},\prod_{j=1}^{n-1} y_j^{b_{j,2}},\ldots,\prod_{j=1}^{n-1} y_j^{b_{j,n-1}}\rp   \prod_{\ell = 1}^{n-1} y_{\ell}^{-\ell(n-\ell)-1}dy_{\ell}.
\eeq
At this moment, we make a change of variables:
\beq
w_{k}&=&\prod_{j=1}^{n-1} y_j^{b_{j,k}}\hspace{10mm}\mbox{     for }1\le k\le n-1.
\eeq
One can easily show that
\beq
y_j = \lp \frac{w_{n-j}^2}{w_{n-j-1}w_{n-j+1}}\rp^{\frac{1}{n}}\hspace{10mm}\mbox{     for }1\le j\le n-1
\eeq
where we set $w_0=w_n=1$. Then we can calculate the Jacobian of this change of variables
\beq
J = \frac{1}{n^{n-2}}(w_1w_{n-1})^{-\frac{n-1}{n}}  \prod_{k=2}^{n-2}w_k^{-1}.
\eeq
At the same time,
\beq
 \prod_{\ell = 1}^{n-1} y_{\ell}^{-\ell(n-\ell)-1} = (w_1w_{n-1})^{-\frac{3}{n}}  \prod_{k=2}^{n-2}w_k^{-\frac{2}{n}}.
\eeq
Thus 
\beq
&&\int_{SL_n(\bbZ)\backslash X_n} E_{(1,\ldots,1)}(z,\eta) d^*z\\
&=&\int_{{(\bbR^{+})}^{n-1}}\eta\lp w_1,\ldots,w_{n-1}\rp  \frac{1}{n^{n-2}}\prod_{k=1}^{n-1}w_k^{-\frac{2}{n}} \prod_{\ell = 1}^{n-1} \frac{dw_{\ell}}{w_{\ell}}\\
&=& \frac{1}{n^{n-2}}\tilde{\eta}\lp\frac{2}{n},\ldots,\frac{2}{n}\rp
\eeq
as desired.
\end{proof}

To prove Theorem \ref{mainterm}, we use the Fourier expansion of the degenerate Eisenstein series from Theorem \ref{prop3}. We cross multiply the Fourier expansion to open the square. There are four type of terms that we will evaluate separately:
\begin{enumerate}
\item \beq\Delta_1(z,t):=\lbar\hat{\phi}_{(0,\ldots,0)}\lp z,\frac{1}{2}+it\rp\rbar^2;\eeq
\item \beq
\Delta_2 (z,t):=\lbar \sum_{m_1\neq0} \hat{\phi}_{(m_1,0,\ldots,0)}\lp z,\frac{1}{2}+it\rp \rbar^2;
\eeq
\item   For $2\le k \le n-1$\beq
\Delta_{3,k} (z,t) := \lbar \sum_{\gamma_i\in P_{i}(\bbZ)\backslash SL_k(\bbZ)} \sum_{m_k=1}^\infty \hat{\phi}_{(0,\ldots,0,m_k,0\ldots,0)}\lp\bemtwo \gamma_i &\\ & I_{n-i}\enm z,\frac{1}{2}+it\rp  \rbar^2;
\eeq
\item All cross terms.
\end{enumerate}

\begin{lem}
All cross terms vanish.
\end{lem}

\begin{proof}

A typical term after opening the square has the form
\beq
\sum_{\gamma_i\in P_{i}(\bbZ)\backslash SL_i(\bbZ)} \sum_{m_i=1}^\infty \hat{\phi}_{(0,\ldots,0,m_i,0\ldots,0)}\lp\bemtwo \gamma_i &\\ & I_{n-i}\enm z\rp   \sum_{\gamma_j\in P_{j}(\bbZ)\backslash SL_j(\bbZ)} \sum_{m_j=1}^\infty \overline{\hat{\phi}}_{(0,\ldots,0,m_j,0\ldots,0)}\lp\bemtwo \gamma_j &\\ & I_{n-j}\enm z\rp.
\eeq

If $i\neq j$, without loss of generality we can assume $i>j$. Then integrating against the variables $x_{1,k+1},\ldots, x_{k-1,k+1}$ forces $\gamma_i$ to be the identity coset. Then the this cross term vanishes after integrating against $x_{k,k+1}$. 

\end{proof}

\begin{lem}
\beq
\int_{SL_n(\bbZ)\backslash X_n} E_{(1,\ldots,1)}(z,\eta) \Delta_1(z,t) d^*z = O(1).
\eeq
\end{lem}
\begin{proof}
\beq
\Delta_1(z,t) = \lbar \sum_{k=0}^{n-1} \frac{2\xi\lp k+1-\frac{n}{2}+int\rp}{\xi\lp\frac{n}{2}+int\rp} \lp y_1 y_2^2\cdots y_{n-k-1}^{n-k-1}\rp^{\frac{1}{2}-it}    \lp y_{n-k}^{k}y_{n-(k-1)}^{k-1}\cdots y_{n-1}\rp^{\frac{1}{2}+it}  \rbar^2.
\eeq
First we treat diagonal terms:
\beq
&&\int_{SL_n(\bbZ)\backslash X_n} E_{(1,\ldots,1)}(z,\eta) \sum_{k=0}^{n-1} \frac{\lbar\xi\lp k+1-\frac{n}{2}+int\rp\rbar^2}{\lbar\xi\lp\frac{n}{2}+int\rp\rbar^2} \lp y_1 y_2^2\cdots y_{n-k-1}^{n-k-1}\rp   \lp y_{n-k}^{k}y_{n-(k-1)}^{k-1}\cdots y_{n-1}\rp d^*z\\
&=& \sum_{k=0}^{n-1} \frac{\lbar\xi\lp k+1-\frac{n}{2}+int\rp\rbar^2}{\lbar\xi\lp\frac{n}{2}+int\rp\rbar^2}  \int_0^\infty\cdots\int_0^\infty\eta\lp \prod_{j=1}^{n-1} y_j^{b_{j,1}},\prod_{j=1}^{n-1} y_j^{b_{j,2}},\ldots,\prod_{j=1}^{n-1} y_j^{b_{j,n-1}}\rp\\
&&\times  \lp y_1 y_2^2\cdots y_{n-k-1}^{n-k-1}\rp   \lp y_{n-k}^{k}y_{n-(k-1)}^{k-1}\cdots y_{n-1}\rp \prod_{\ell = 1}^{n-1} y_{\ell}^{-\ell(n-\ell)-1}dy_{\ell}\\
&=&\sum_{k=0}^{n-1} \frac{\lbar\xi\lp k+1-\frac{n}{2}+int\rp\rbar^2}{\lbar\xi\lp\frac{n}{2}+int\rp\rbar^2}  c_{k,\eta}
\eeq
for some constants $c_{k,\eta}$. We will not calculate these constants as they do not contribute to the main term. For $k=0,n-1$, we get a contribution of $O(1)$. By Stirling's formula the rest will contribute $O(t^{-1})$.

For off-diagonal terms, some powers of $y$ will be imaginary unlike the diagonal terms. We can then use integration by parts repeatedly and get a contribution of $O(t^{-2016})$, say.
\end{proof}

\begin{lem}
\beq
&&\int_{SL_n(\bbZ)\backslash X_n} E_{(1,\ldots,1)}(z,\eta) \Delta_2(z,t) d^*z \\
&=&\frac{2}{\lbar\xi\lp\frac{n}{2}+int\rp\rbar^2}\int_{(\bbR^+)^{n-2}} \frac{1}{(2\pi i)^{n-1}} \int_{(2)}\cdots \int_{(2)}\tilde{\eta}(\nu)     \\
&&\times   \frac{\xi\lp {\sum_{j=1}^{n-1}b_{n-1,j}\nu_j}  \rp\xi\lp {\sum_{j=1}^{n-1}b_{n-1,j}\nu_j} -\frac{n}{2}+1+int \rp }{\xi\lp 2{\sum_{j=1}^{n-1}b_{n-1,j}\nu_j}  -n+2 \rp}\\
&&\times \xi\lp  {\sum_{j=1}^{n-1}b_{n-1,j}\nu_j} -\frac{n}{2}+1-int  \rp\xi\lp \sum_{j=1}^{n-1}b_{n-1,j}\nu_j-n+2\rp\\
&&\times \lp y_1 y_2^2\cdots y_{n-2}^{n-2}\rp     \prod_{\ell=1}^{n-1}d\nu_{\ell}   \prod_{k=1}^{n-2}\frac{dy_k}{y_k^{k(n-k)+1}}
\eeq
\end{lem}
\begin{proof}
\beq
&&\int_{SL_n(\bbZ)\backslash X_n} E_{(1,\ldots,1)}(z,\eta) \Delta_2(z,t) d^*z\\
&=& \int_{SL_n(\bbZ)\backslash X_n} E_{(1,\ldots,1)}(z,\eta) \lbar \sum_{m_1\neq0} \hat{\phi}_{(m_1,0,\ldots,0)}\lp z,\frac{1}{2}+it\rp \rbar^2d^*z\\
&=&\frac{8}{\lbar\xi\lp\frac{n}{2}+int\rp\rbar^2} \int_{(\bbR^+)^{n-1}} \int_{N_{min(\bbZ)}\backslash N_{min(\bbR)}}  \frac{1}{(2\pi i)^{n-1}} \int_{(2)}\cdots \int_{(2)}\tilde{\eta}(\nu) \lp \prod_{i=1}^{n-1} y_i^{\sum_{j=1}^{n-1}b_{i,j}\nu_j} \rp\\
&&\times   \sum_{m_1=1}^\infty m_1^{1-\frac{n}{2}}  \sigma_{\frac{n}{2}-1-int}(m_1)\sigma_{\frac{n}{2}-1+int}(m_1) K_{\frac{1}{2}-\frac{n}{4}+\frac{int}{2}}(2\pi m_1y_{n-1})K_{\frac{1}{2}-\frac{n}{4}-\frac{int}{2}}(2\pi m_1y_{n-1}) \\
&&\times   \lp y_1 y_2^2\cdots y_{n-2}^{n-2}\rp   y_{n-1}^{\frac{n}{2}}  \prod_{\ell=1}^{n-1}d\nu_{\ell}   \prod_{k=1}^{n-1}\frac{dy_k}{y_k^{k(n-k)+1}}
\eeq
The Lemma is then proved by applying Lemma \ref{divisor} to convert the divisor sum into Zeta functions and \ref{stade} to convert K-Bessel integrals to Gamma functions.
\end{proof}

\begin{lem}
\beq&&\int_{SL_n(\bbZ)\backslash X_n} E_{(1,\ldots,1)}(z,\eta) \Delta_{3,k}(z,t) d^*z \\
&=&\frac{2}{\lbar\xi\lp\frac{n}{2}+int\rp\rbar^2}\int_{(\bbR^+)^{n-2}} \frac{1}{(2\pi i)^{n-1}} \int_{(2)}\cdots \int_{(2)}\tilde{\eta}(\nu)     \\
&&\times\sum_{h=0}^{k-1}  \frac{\xi\lp (1-k)(n-k)+{\sum_{j=1}^{n-1}b_{n-k,j}\nu_j}  \rp\xi\lp (1-k)(n-k)+{\sum_{j=1}^{n-1}b_{n-k,j}\nu_j} -\frac{n}{2}+k+int \rp}{\xi\lp  2(1-k)(n-k)+2{\sum_{j=1}^{n-1}b_{n-k,j}\nu_j}  -n+2k \rp}\\
&&\times \xi\lp (1-k)(n-k)+{\sum_{j=1}^{n-1}b_{n-k,j}\nu_j} -\frac{n}{2}+k-int  \rp\xi\lp \sum_{j=1}^{n-1}b_{n-k,j}\nu_j-k(n-k) +h+1\rp\\
&&\times \prod_{i=1}^{n-k-1} y_i^{\sum_{j=1}^{n-1} b_{i,j}\nu_j+i-i(n-i)} \prod_{i=n-k+1}^{n-h-1} y_i^{k(n-k)-i(n-i)+\sum_{j=1}^{n-1}(b_{i,j}-b_{n-k,j})\nu_j}  \prod_{i=n-h}^{n-1}y_i^{\sum_{j=1}^{n-1}b_{i,j}\nu_j+n-i-i(n-i)}\\
&&\prod_{\ell=1}^{n-1}d\nu_{\ell} \prod_{\substack{i=1\\ i\neq n-k}}^{n-1}\frac{dy_i}{y_i}
\eeq
\end{lem}

\begin{proof}

By Lemma \ref{action}, as $\gamma_i$ maps $x_{i,i+1}$ to $a_1x_{1,i+1}+a_2x_{2,i+1}+\cdots a_i x_{i,i+1}$ where the primitive vector $(a_1,\ldots,a_i)$ is the bottom row of $\gamma_i$. Now integrating the exponential factor 
\beq
e\lp m_i(a_1x_{1,i+1}+a_2x_{2,i+1}+\cdots a_i x_{i,i+1}) - m_i'(a_1'x_{1,i+1}+a_2'x_{2,i+1}+\cdots a_i' x_{i,i+1})\rp
\eeq
in the diagonal term forces $m_i=m_i'$ and $\gamma_i=\gamma_i'$ up to a sign. Thus the $k$-th ($2\le k\le n-1$) diagonal term $\Xi_k$ reads with $s=\frac{1}{2}+it$:
\beq
\Xi_k&:=&\frac{8}{\xi\lp\frac{n}{2}+int\rp\xi\lp\frac{n}{2}-int\rp} \sum_{\gamma_k\in P_{k}(\bbZ)\backslash SL_k(\bbZ)} \sum_{m_k=1}^\infty |m_k|^{k-\frac{n}{2}}  \sigma_{\frac{n}{2}-k-int}(|m_k|)\sigma_{\frac{n}{2}-k+int}(|m_k|)\\
&&\times K_{\frac{k}{2}-\frac{n}{4}+\frac{int}{2}}(2\pi |m_k|y_{n-k}')K_{\frac{k}{2}-\frac{n}{4}-\frac{int}{2}}(2\pi |m_k|y_{n-k}')   \lp y_1 y_2^2\cdots y_{n-k-1}^{n-k-1}\rp   \lp y_{n-k}^{k}y_{n-(k-1)}^{k-1}\cdots y_{n-1}\rp y_{n-k}'^{\frac{n}{2}-k}
\eeq

Putting all terms together, we get
\beq
\Omega_k:=\int_{(\bbR^+)^{n-1}} \int_{N_{min(\bbZ)}\backslash N_{min(\bbR)}}  \frac{1}{(2\pi i)^{n-1}} \int_{(2)}\cdots \int_{(2)} \prod_{i=1}^{n-1} y_i^{\sum_{j=1}^{n-1}b_{i,j}\nu_j} \Xi_k \prod_{\ell=1}^{n-1}d\nu_{\ell} \prod_{1\le i<j\le n-1}dx_{i,j} \prod_{k=1}^{n-1}\frac{dy_k}{y_k^{k(n-k)+1}}.
\eeq
We now make the change of variable $|m_k|y_{n-k}' \mapsto y_{n-k}$ and let $\Theta_k=(b_1^2+b_2^2+\cdots+b_{k}^2)^{\frac{1}{2}}$. So we have
\beq
\Omega_k &=& \frac{8}{\xi\lp\frac{n}{2}+int\rp\xi\lp\frac{n}{2}-int\rp} \int_{(\bbR^+)^{n-1}} \int_{N_{min(\bbZ)}\backslash N_{min(\bbR)}}  \frac{1}{(2\pi i)^{n-1}} \int_{(2)}\cdots \int_{(2)}\tilde{\eta}(\nu) \lp \prod_{i=1}^{n-1} y_i^{\sum_{j=1}^{n-1}b_{i,j}\nu_j} \rp\\
&&\times \sum_{\gamma_k\in P_{k}(\bbZ)\backslash SL_k(\bbZ)} \sum_{m_k=1}^\infty |m_k|^{k-\frac{n}{2}}  \sigma_{\frac{n}{2}-k-int}(|m_k|)\sigma_{\frac{n}{2}-k+int}(|m_k|)\lp |m_k|\Theta_k\rp^{-\sum_{j=1}^{n-1}b_{n-k,j}\nu_j-k+k(n-k)}|m_k|^{k-\frac{n}{2}}\\
&&\times K_{\frac{k}{2}-\frac{n}{4}+\frac{int}{2}}(2\pi y_{n-k})K_{\frac{k}{2}-\frac{n}{4}-\frac{int}{2}}(2\pi y_{n-k})   \lp y_1 y_2^2\cdots y_{n-k-1}^{n-k-1}\rp   \lp y_{n-k}^{k}y_{n-(k-1)}^{k-1}\cdots y_{n-1}\rp y_{n-k}^{\frac{n}{2}-k}\\
&& \prod_{\ell=1}^{n-1}d\nu_{\ell} \prod_{1\le i<j\le n-1}dx_{i,j} \prod_{k=1}^{n-1}\frac{dy_k}{y_k^{k(n-k)+1}}.
\eeq
We use the identity
\beq
\sum_{n=1}\frac{\sigma_a(n)\sigma_b(n)}{n^s}= \frac{\zeta(s)\zeta(s-a)\zeta(s-b)\zeta(s-a-b)}{\zeta(2s-a-b)}.
\eeq
to convert the sum over $m_k$ to $\zeta$ functions:

\beq
&&\sum_{m_k=1}^\infty \frac{  \sigma_{\frac{n}{2}-k-int}(|m_k|)\sigma_{\frac{n}{2}-k+int}(|m_k|)}{|m_k|^{n-k-k(n-k)+{\sum_{j=1}^{n-1}b_{n-k,j}\nu_j}}}\\
&=&\frac{\zeta\lp (1-k)(n-k)+{\sum_{j=1}^{n-1}b_{n-k,j}\nu_j}  \rp\zeta\lp (1-k)(n-k)+{\sum_{j=1}^{n-1}b_{n-k,j}\nu_j} -\frac{n}{2}+k+int \rp}{\zeta\lp  2(1-k)(n-k)+2{\sum_{j=1}^{n-1}b_{n-k,j}\nu_j}  -n+2k \rp}\\
&&\times \zeta\lp (1-k)(n-k)+{\sum_{j=1}^{n-1}b_{n-k,j}\nu_j} -\frac{n}{2}+k-int  \rp\zeta\lp (1-k)(n-k)+{\sum_{j=1}^{n-1}b_{n-k,j}\nu_j}  -n+2k \rp.
\eeq

Now we notice that 
\beq
\sum_{\gamma_k\in P_{k}(\bbZ)\backslash SL_k(\bbZ)}    \Theta_k^{-\sum_{j=1}^{n-1}b_{n-k,j}\nu_j-k+k(n-k)}
\eeq
is just a degenerate Eisenstein series. Integrating over $x$ simply gives the constant term. It reads
\beq
\sum_{h=0}^{k-1}\frac{2\xi\lp \sum_{j=1}^{n-1}b_{n-k,j}\nu_j-k(n-k) +h+1\rp           \lp    y_{(n-k)+1}\cdots y_{(n-k)+(k-h-1)^+}  \rp^{- \sum_{j=1}^{n-1}b_{n-k,j}\nu_j+k(n-k) -h-1}}{\xi\lp \sum_{j=1}^{n-1}b_{n-k,j}\nu_j-k(n-k) +k\rp     \lp y_{(n-k)+1}^{k-h-2}\cdots y_{(n-k)+(k-h-2)^+}  \rp}
\eeq

Now we use the identity 
 \beq
 \int_0^\infty K_\mu(y)K_\nu(y)y^s\frac{dy}{y} = 2^{s-3}\frac{\Gamma\lp\frac{s+\mu+\nu}{2}\rp\Gamma\lp\frac{s+\mu-\nu}{2}\rp\Gamma\lp\frac{s-\mu+\nu}{2}\rp\Gamma\lp\frac{s-\mu-\nu}{2}\rp}{\Gamma(s)}
 \eeq
to evaluate the integral of product $K$-Bessel functions:
\beq
\int_0^\infty    K_{\frac{k}{2}-\frac{n}{4}+\frac{int}{2}}(2\pi y_{n-k})K_{\frac{k}{2}-\frac{n}{4}-\frac{int}{2}}(2\pi y_{n-k})    y_{n-k}^{ \sum_{j=1}^{n-1}b_{n-k,j}\nu_j +\frac{n}{2}-k(n-k)}  \frac{d y_{n-k}}{y_{n-k}}.
\eeq
These $\gamma$ factors will form completed $\zeta$-functions:
\beq
&&\frac{1}{8}\frac{\xi\lp (1-k)(n-k)+{\sum_{j=1}^{n-1}b_{n-k,j}\nu_j}  \rp\xi\lp (1-k)(n-k)+{\sum_{j=1}^{n-1}b_{n-k,j}\nu_j} -\frac{n}{2}+k+int \rp}{\xi\lp  2(1-k)(n-k)+2{\sum_{j=1}^{n-1}b_{n-k,j}\nu_j}  -n+2k \rp}\\
&&\times \xi\lp (1-k)(n-k)+{\sum_{j=1}^{n-1}b_{n-k,j}\nu_j} -\frac{n}{2}+k-int  \rp\xi\lp (1-k)(n-k)+{\sum_{j=1}^{n-1}b_{n-k,j}\nu_j}  -n+2k \rp.
\eeq
Now we gather $y_i$ factors for $i\neq n-k$:
\beq
&&\frac{   \lp \prod_{\substack{i=1\\i\neq n-k}}^{n-1} y_i^{\sum_{j=1}^{n-1}b_{i,j}\nu_j} \rp     \lp y_1 y_2^2\cdots y_{n-k-1}^{n-k-1}\rp   \lp y_{n-(k-1)}^{k-1}\cdots y_{n-1}\rp    }{ \lp y_{(n-k)+1}^{k-h-2}\cdots y_{(n-k)+(k-h-2)^+} \rp  \prod_{\substack{\ell=1\\\ell\neq n-k}}^{n-1}{y_\ell^{\ell(n-\ell)}} }\\
&&\times  \lp    y_{(n-k)+1}\cdots y_{(n-k)+(k-h-1)^+}  \rp^{- \sum_{j=1}^{n-1}b_{n-k,j}\nu_j+k(n-k) -h-1}\\
&=&\prod_{i=1}^{n-k-1} y_i^{\sum_{j=1}^{n-1} b_{i,j}\nu_j+i-i(n-i)} \prod_{i=n-k+1}^{n-h-1} y_i^{k(n-k)-i(n-i)+\sum_{j=1}^{n-1}(b_{i,j}-b_{n-k,j})\nu_j}  \prod_{i=n-h}^{n-1}y_i^{\sum_{j=1}^{n-1}b_{i,j}\nu_j+n-i-i(n-i)}.
\eeq

Now we put everything together, 
\beq
\Omega_k &=&\frac{2}{\lbar\xi\lp\frac{n}{2}+int\rp\rbar^2}\int_{(\bbR^+)^{n-2}} \frac{1}{(2\pi i)^{n-1}} \int_{(2)}\cdots \int_{(2)}\tilde{\eta}(\nu)     \\
&&\times\sum_{h=0}^{k-1}  \frac{\xi\lp (1-k)(n-k)+{\sum_{j=1}^{n-1}b_{n-k,j}\nu_j}  \rp\xi\lp (1-k)(n-k)+{\sum_{j=1}^{n-1}b_{n-k,j}\nu_j} -\frac{n}{2}+k+int \rp}{\xi\lp  2(1-k)(n-k)+2{\sum_{j=1}^{n-1}b_{n-k,j}\nu_j}  -n+2k \rp}\\
&&\times \xi\lp (1-k)(n-k)+{\sum_{j=1}^{n-1}b_{n-k,j}\nu_j} -\frac{n}{2}+k-int  \rp\xi\lp \sum_{j=1}^{n-1}b_{n-k,j}\nu_j-k(n-k) +h+1\rp\\
&&\times \prod_{i=1}^{n-k-1} y_i^{\sum_{j=1}^{n-1} b_{i,j}\nu_j+i-i(n-i)} \prod_{i=n-k+1}^{n-h-1} y_i^{k(n-k)-i(n-i)+\sum_{j=1}^{n-1}(b_{i,j}-b_{n-k,j})\nu_j}  \prod_{i=n-h}^{n-1}y_i^{\sum_{j=1}^{n-1}b_{i,j}\nu_j+n-i-i(n-i)}\\
&&\prod_{\ell=1}^{n-1}d\nu_{\ell} \prod_{\substack{i=1\\ i\neq n-k}}^{n-1}\frac{dy_i}{y_i}
\eeq
\end{proof}
Combining the previous two lemmas, we let
\beq 
\Omega_{k,h}&:=&\frac{2}{\lbar\xi\lp\frac{n}{2}+int\rp\rbar^2}\int_{(\bbR^+)^{n-2}} \frac{1}{(2\pi i)^{n-1}} \int_{(2)}\cdots \int_{(2)}\tilde{\eta}(\nu)     \\
&&\times   \frac{\xi\lp (1-k)(n-k)+{\sum_{j=1}^{n-1}b_{n-k,j}\nu_j}  \rp\xi\lp (1-k)(n-k)+{\sum_{j=1}^{n-1}b_{n-k,j}\nu_j} -\frac{n}{2}+k+int \rp}{\xi\lp  2(1-k)(n-k)+2{\sum_{j=1}^{n-1}b_{n-k,j}\nu_j}  -n+2k \rp}\\
&&\times \xi\lp (1-k)(n-k)+{\sum_{j=1}^{n-1}b_{n-k,j}\nu_j} -\frac{n}{2}+k-int  \rp\xi\lp \sum_{j=1}^{n-1}b_{n-k,j}\nu_j-k(n-k) +h+1\rp\\
&&\times \prod_{i=1}^{n-k-1} y_i^{\sum_{j=1}^{n-1} b_{i,j}\nu_j+i-i(n-i)} \prod_{i=n-k+1}^{n-h-1} y_i^{k(n-k)-i(n-i)+\sum_{j=1}^{n-1}(b_{i,j}-b_{n-k,j})\nu_j}  \prod_{i=n-h}^{n-1}y_i^{\sum_{j=1}^{n-1}b_{i,j}\nu_j+n-i-i(n-i)}\\
&&\prod_{\ell=1}^{n-1}d\nu_{\ell} \prod_{\substack{i=1\\ i\neq n-k}}^{n-1}\frac{dy_i}{y_i}
\eeq
for $1\le k\le n-1$ and $0\le h\le k-1$, then we claim that the main term comes from $\Omega_{n-1,0}$. We have
\beq
\Omega_{n-1,0}&=&\frac{2}{\lbar\xi\lp\frac{n}{2}+int\rp\rbar^2}\int_{(\bbR^+)^{n-2}} \frac{1}{(2\pi i)^{n-1}} \int_{(2)}\cdots \int_{(2)}\tilde{\eta}(\nu)     \\
&&\times   \frac{\xi\lp 2-n+{\sum_{j=1}^{n-1}b_{1,j}\nu_j}  \rp^2 \xi\lp 1-\frac{n}{2}+{\sum_{j=1}^{n-1}b_{n-k,j}\nu_j} +int \rp  \xi\lp 1 -\frac{n}{2}+{\sum_{j=1}^{n-1}b_{n-k,j}\nu_j}-int  \rp}{\xi\lp  2-n+2{\sum_{j=1}^{n-1}b_{n-k,j}\nu_j}    \rp}\\
&&\times  \prod_{i=2}^{n-1} y_i^{n-1-i(n-i)+\sum_{j=1}^{n-1}(b_{i,j}-b_{1,j})\nu_j}  \prod_{\ell=1}^{n-1}d\nu_{\ell} \prod_{\substack{i=1\\ i\neq n-k}}^{n-1}\frac{dy_i}{y_i}
\eeq

\begin{lem}
\beq
\Omega_{n-1,0}  = \frac{ 2 \tilde{\eta}\lp\frac{2}{n},\ldots,\frac{2}{n}\rp  \log t}{\xi(n)n^{n-2}} +O(1).
\eeq
\end{lem}
\begin{proof}
Notice that there are $n-2$ Mellin transforms and $n-1$ Mellin inversion in the above expression. Thus we can take advantage of the Mellin inversion theorem. We make the change of variables
\begin{eqnarray}
\sum_{j=1}^{n-1}(b_{i,j}-b_{1,j})\nu_j \mapsto s_{i-1}\ \ \mbox{ for } 2\le i\le  n-1   \label{change}
\end{eqnarray}
and $s_{n-1}=v_{n-1}$. The Jacobian is calculated to be $\frac{1}{n^{n-3}}$ and explicitly we have
\beq
v_1 &=& \frac{s_1 - s_{n-3}+2s_{n-2}}{n} + v_{n-1},\\
v_i &=& \frac{2s_{n-i-1}-s_{n-i}-s_{n-i-2}}{n}\ \ \mbox{ for } 2\le i\le n-3,\\
v_{n-2}&=&\frac{2s_1-s_2}{8}\\
v_{n-1}&=&s_{n-1}.
\eeq
By \ref{change}, we also have $\sum_{j=1}^{n-1}b_{1,j}\nu_j \mapsto s_1+n s_{n-1}$. So 
\beq
\Omega_{n-1,0}&=&\frac{1}{n^{n-3}}\frac{2}{\lbar\xi\lp\frac{n}{2}+int\rp\rbar^2}\int_{(\bbR^+)^{n-2}} \frac{1}{(2\pi i)^{n-1}} \int_{(\sigma_1)}\cdots \int_{(\sigma_{n-1})}\\&&\tilde{\eta}\lp \frac{s_1 - s_{n-3}+2s_{n-2}}{n} + s_{n-1} , \ldots, \frac{2s_{n-i-1}-s_{n-i}-s_{n-i-2}}{n},\ldots, \frac{2s_1-s_2}{n},s_{n-1}\rp     \\
&&\times   \frac{\xi\lp 2-n+{s_1+n s_{n-1}}  \rp^2 \xi\lp 1-\frac{n}{2}+{s_1+n s_{n-1}} +int \rp  \xi\lp 1 -\frac{n}{2}+{s_1+n s_{n-1}}-int  \rp}{\xi\lp  2-n+{2s_1+2n s_{n-1}}    \rp}\\
&&\times  \prod_{i=2}^{n-1} y_i^{n-1-i(n-i)+s_{i-1}}  \prod_{\ell=1}^{n-1}ds_{\ell} \prod_{\substack{i=1\\ i\neq n-k}}^{n-1}\frac{dy_i}{y_i}\\
&=&\frac{1}{n^{n-3}}\frac{2}{\lbar\xi\lp\frac{n}{2}+int\rp\rbar^2} \frac{1}{2\pi i}  \int_{(\sigma_{n-1})} \tilde{\eta}\lp   s_{n-1} ,\frac{2}{n}, \ldots, \frac{2}{n},s_{n-1}\rp     \\
&&\times   \frac{\xi\lp {n s_{n-1}-1}  \rp^2 \xi\lp \frac{n}{2}-2+{n s_{n-1}} +int \rp  \xi\lp \frac{n}{2}-2+n s_{n-1}-int  \rp}{\xi\lp  n-4+2n s_{n-1}    \rp}\  ds_{n-1}
\eeq

We now shift contour from $\sigma_{n-1}$ to $\frac{2-\Delta}{n}$ for some small $\Delta>0$. In the process, we pick up the residue corresponding to the pole of the $\zeta(ns_{n-1}-1)^2$ term at $s_{n-1} = \frac{2}{n}$. Hence
\beq
\Omega_{n-1,0}&=&R_1+R_2
\eeq
where
\beq
R_1&:=&\frac{ 2 \tilde{\eta}\lp\frac{2}{n},\ldots,\frac{2}{n}\rp}{\xi(n)n^{n-1}}\lp O(1)+\frac{n}{2}\lp\frac{\Gamma'}{\Gamma}\lp\frac{n}{4}+\frac{int}{2}\rp+\frac{\Gamma'}{\Gamma}\lp\frac{n}{4}-\frac{int}{2}\rp\rp\rp\\
&=&\frac{ 2 \tilde{\eta}\lp\frac{2}{n},\ldots,\frac{2}{n}\rp  \log t}{\xi(n)n^{n-2}} +O(1),
\eeq
\beq
R_2&:=&\frac{1}{n^{n-3}}\frac{2}{\lbar\xi\lp\frac{n}{2}+int\rp\rbar^2} \frac{1}{2\pi i}  \int_{\lp\frac{2-\Delta}{n}\rp} \tilde{\eta}\lp   s_{n-1} ,\frac{2}{n}, \ldots, \frac{2}{n},s_{n-1}\rp     \\
&&\times   \frac{\xi\lp {n s_{n-1}-1}  \rp^2 \xi\lp \frac{n}{2}-2+{n s_{n-1}} +int \rp  \xi\lp \frac{n}{2}-2+n s_{n-1}-int  \rp}{\xi\lp  n-4+2n s_{n-1}    \rp}\  ds_{n-1}.
\eeq
Let $f(t_2) = \tilde{\eta}\lp \frac{2-\Delta}{n}+it_2,\frac{2}{n},\ldots,\frac{2}{n},\frac{2-\Delta}{n}+it_2\rp$, then $f(t_s)$ is a function with rapid decay in $t_2$. We can bound the $R_2$ integral as follows:
\beq
R_2&\ll& \int_{-\infty}^\infty |f(t_2)|\frac{\lbar\zeta\lp 1-\Delta+int_2\rp\rbar^2\lbar\Gamma\lp\frac{1-\Delta}{2}+\frac{int_2}{2}\rp\rbar^2}{\zeta\lp n-2\Delta+2int_2\rp\Gamma\lp \frac{n}{2}-\Delta+int_2\rp}\\
&&\times \frac{ \lbar\zeta\lp\frac{n}{2}-\Delta +in(t_2+t)\rp\rbar \lbar\zeta\lp\frac{n}{2}-\Delta +in(t_2-t)\rp\rbar  \lbar\Gamma\lp \frac{n}{4}-\frac{\Delta}{2}+\frac{in(t_2+t)}{2}\rp\rbar  \lbar\Gamma\lp \frac{n}{4}-\frac{\Delta}{2}+\frac{in(t_2-t)}{2}\rp\rbar}{ \lbar\zeta\lp\frac{n}{2}+int\rp\rbar^2\lbar\Gamma\lp\frac{n}{4}+\frac{int}{2}\rp\rbar^2} dt_2
\eeq

By Stirling's formula, 
\beq
&&\frac{ \lbar\Gamma\lp \frac{n}{4}-\frac{\Delta}{2}+\frac{in(t_2+t)}{2}\rp\rbar \lbar\Gamma\lp \frac{n}{4}-\frac{\Delta}{2}+\frac{in(t_2-t)}{2}\rp\rbar}{\lbar\Gamma\lp\frac{n}{4}+\frac{int}{2}\rp\rbar^2}\\
&\ll&\frac{ e^{-\frac{\pi n}{4}\lp|t_2+t|+|t_2-t|\rp} \lbar t_2+t\rbar^{\frac{n}{4}-\frac{1}{2}-\frac{\Delta}{2}}\lbar t_2-t\rbar^{\frac{n}{4}-\frac{1}{2}-\frac{\Delta}{2}}}{e^{-\frac{\pi nt}{2}} |t|^{\frac{n}{2}-1}}\\
&\ll& \frac{\lbar t_2+t\rbar^{\frac{n}{4}-\frac{1}{2}-\frac{\Delta}{2}}\lbar t_2-t\rbar^{\frac{n}{4}-\frac{1}{2}-\frac{\Delta}{2}}}{|t|^{\frac{n}{2}-1}}.
\eeq
So
\beq
R_2\ll t^{-\Delta}.
\eeq 

\end{proof}

Now we tackle $\Omega_{k,h}$ for general $k,h$. \begin{prop}
For $(k,h)\neq (n-1,0)$, we have
\beq
\Omega_{k,h} = O(1).
\eeq
\end{prop}
\begin{proof}

We make the convenient change of variable
\beq
\bemone s_1\\ s_2\\ \vdots\\ s_{n-1}\enm = M_{k,h} \cdot \lp b_{i,j}\rp \cdot \bemone v_1\\ v_2\\ \vdots\\ v_{n-1}\enm 
\eeq
where the matrix $M_{k,h}$ is the identity matrix except for the entries $(j,n-k)$ being $-1$ for $n-k+1\le j\le n-h-1$. With this change variable we have
\beq
\Omega_{k,h}&=&\frac{2}{ n^{n-2}\lbar\xi\lp\frac{n}{2}+int\rp\rbar^2}\int_{(\bbR^+)^{n-2}} \frac{1}{(2\pi i)^{n-1}} \int_{(\sigma_1)}\cdots \int_{(\sigma_{n-1})}\tilde{\eta}(v_1,\ldots,v_{n-1})     \\
&&\times   \frac{\xi\lp (1-k)(n-k)+s_{n-k} \rp   \xi\lp s_{n-k}-k(n-k) +h+1\rp  }{\xi\lp  2(1-k)(n-k)+2s_{n-k} -n+2k \rp}\\
&&\times \xi\lp (1-k)(n-k)+s_{n-k} -\frac{n}{2}+k-int  \rp   \xi\lp (1-k)(n-k)+s_{n-k} -\frac{n}{2}+k+int \rp \\
&&\times \prod_{i=1}^{n-k-1} y_i^{s_i+i-i(n-i)} \prod_{i=n-k+1}^{n-h-1} y_i^{k(n-k)-i(n-i)+s_i}  \prod_{i=n-h}^{n-1}y_i^{s_i+n-i-i(n-i)}\prod_{\ell=1}^{n-1}ds_{\ell} \prod_{\substack{i=1\\ i\neq n-k}}^{n-1}\frac{dy_i}{y_i}\\
&=&\frac{2}{ n^{n-2}\lbar\xi\lp\frac{n}{2}+int\rp\rbar^2} \frac{1}{2\pi i} \int_{(\sigma_{n-1})}f(s_{n-1})     \frac{\xi\lp (1-k)(n-k)+s_{n-k} \rp   \xi\lp s_{n-k}-k(n-k) +h+1\rp  }{\xi\lp  2(1-k)(n-k)+2s_{n-k} -n+2k \rp}\\
&&\times \xi\lp (1-k)(n-k)+s_{n-k} -\frac{n}{2}+k-int  \rp   \xi\lp (1-k)(n-k)+s_{n-k} -\frac{n}{2}+k+int \rp ds_{n-1}.
\eeq
Now we will shift contours to $\re( s_{n-k} )= 1-(1-k)(n-k)-\Delta$ for small $\Delta>0$. There are two cases to consider.

{\bf Case 1:} $n-k=h+1$ and $k\neq n-1$.

In this situation, $\xi\lp (1-k)(n-k)+s_{n-k} \rp  = \xi\lp s_{n-k}-k(n-k) +h+1\rp$. As a result, we will encounter a pole of order $2$ during the contour shift at $s_{n-1} = 1-(1-k)(n-k)$. So
\beq
\Omega_{k,h}=R_1+R_2
\eeq
where
\beq
R_1 = \frac{C_1\xi\lp k+1-\frac{n}{2}\rp \xi\lp k+1-\frac{n}{2}\rp}{\lbar\xi\lp \frac{n}{2}+int\rp\rbar^2}\lp C_2+\frac{\xi'}{\xi}\lp1+k-\frac{n}{2}+int\rp+\frac{\xi'}{\xi}\lp1+k-\frac{n}{2}-int\rp\rp
\eeq
and 
\beq
R_2& =&\frac{2}{ n^{n-2}\lbar\xi\lp\frac{n}{2}+int\rp\rbar^2} \frac{1}{2\pi i} \int_{(1-(1-k)(n-k)-\Delta)}f(s_{n-1})     \frac{\xi\lp (1-k)(n-k)+s_{n-k} \rp^2  }{\xi\lp  2(1-k)(n-k)+2s_{n-k} -n+2k \rp}\\
&&\times \xi\lp (1-k)(n-k)+s_{n-k} -\frac{n}{2}+k-int  \rp   \xi\lp (1-k)(n-k)+s_{n-k} -\frac{n}{2}+k+int \rp ds_{n-1}.
\eeq
By Stirling's formula, one can show that
\beq
R_1+R_2\ll t^{-\frac{1}{2}}.
\eeq
{\bf Case 2:} $n-k\neq h+1$.

For this case, $\xi\lp (1-k)(n-k)+s_{n-k} \rp  \neq \xi\lp s_{n-k}-k(n-k) +h+1\rp$. As a result, we will encounter a pole of order $1$ during the contour shift at $s_{n-1} = 1-(1-k)(n-k)$. So
\beq
\Omega_{k,h}=R_3+R_4
\eeq
where
\beq
R_3 = C_3   \frac{\xi\lp 1 -\frac{n}{2}+k+int\rp \xi\lp 1 -\frac{n}{2}+k-int\rp}{\lbar\xi\lp\frac{n}{2}+int\rp\rbar^2}
\eeq
and 
\beq
R_4&=&\frac{2}{ n^{n-2}\lbar\xi\lp\frac{n}{2}+int\rp\rbar^2} \frac{1}{2\pi i} \int_{(1-(1-k)(n-k)-\Delta)}f(s_{n-1})     \frac{\xi\lp (1-k)(n-k)+s_{n-k} \rp   \xi\lp s_{n-k}-k(n-k) +h+1\rp  }{\xi\lp  2(1-k)(n-k)+2s_{n-k} -n+2k \rp}\\
&&\times \xi\lp (1-k)(n-k)+s_{n-k} -\frac{n}{2}+k-int  \rp   \xi\lp (1-k)(n-k)+s_{n-k} -\frac{n}{2}+k+int \rp ds_{n-1}.
\eeq

By Stirling's formula, one can show that
\beq
R_3+R_4 = O(1).
\eeq

\end{proof}

\end{document}